\documentclass[a4paper,10pt]{article}

\usepackage[utf8]{inputenc}
\usepackage[T1]{fontenc}
\usepackage[english]{babel}

\usepackage{hyperref}
\usepackage{amsthm}
\usepackage{amsmath}
\usepackage{amssymb}
\usepackage{mathrsfs}
\usepackage{bbold}
\usepackage{graphicx}
\usepackage[top=2cm, bottom=2cm, left=2.4cm, right=2.4cm]{geometry}
\usepackage{listings}
\usepackage{mathrsfs}
\usepackage{enumerate}
\usepackage{pict2e}
\usepackage{stmaryrd}

\newtheorem{defi}{Definition}[section]
\newtheorem{theo}[defi]{Theorem}
\newtheorem{propo}[defi]{Proposition}

\newtheorem{lem}[defi]{Lemma}
\newtheorem{corol}[defi]{Corollary}
\newtheorem{rem}[defi]{Remark}

\DeclareUnicodeCharacter{25FA}{\trianglerectangle}

\DeclareMathOperator{\matr}{mat}

\DeclareMathOperator{\Res}{Res}

\DeclareMathOperator{\ord}{ord}
\newcommand*\pFqskip{8mu}
\catcode`,\active
\newcommand*\pFq{\begingroup
        \catcode`\,\active
        \def ,{\mskip\pFqskip\relax}%
        \dopFq
}
\catcode`\,12
\def\dopFq#1#2#3#4#5{%
        {}_{#1}F_{#2}\biggl[\genfrac..{0pt}{}{#3}{#4};#5\biggr]%
        \endgroup
}

\begin{document}
\title{Optimality and resonances in a class of compact finite difference schemes of high order}

\author{Joackim Bernier }

\maketitle

\begin{abstract}
In this paper, we revisit the old problem of compact finite difference approximations of the homogeneous Dirichlet problem in dimension $1$. 
We design a large and natural set of schemes of arbitrary high order, and we equip this set with an algebraic structure. 
We give some general criteria of convergence and we apply them to obtain two new results. On the one hand, we use Pad\'e approximant theory to construct, for each given order of consistency, the {\em most efficient} schemes and we prove their convergence. On the other hand, we use diophantine approximation theory to prove that {\em almost all} of these schemes are convergent at the same rate as the consistency order, up to some logarithmic correction.
\end{abstract}

\tableofcontents

\section{Introduction}
Many decades ago, compact finite differences methods were widely studied. Nowadays, we can find a huge literature about these methods that are widely applied and used for the approximation of partial differential equations (see, for example, \cite{MR0165702} or \cite{MR2124284}). In particular, we can find a lot of examples of accurate schemes for elliptic problems and many classical mathematical arguments are proposed to prove their convergence (monotonicity, energy, green functions, ...). However, it seems that there is not general and algebraic study of compact finite difference schemes for elliptic problems, equivalent to what we can find, for example, for the Runge Kutta methods applied to Cauchy problems (general stability criteria, algebraic order conditions using Hopf algebras and trees as we can see in \cite{MR1227985} or \cite{MR2840298}). 

\medskip

As the field of elliptic problems is clearly too wide, we propose, in this paper, a general study of a large and natural class of compact finite difference schemes of high order for the homogeneous Dirichlet problem in dimension $1$. In this context, a compact finite difference scheme is a linear system of the form
\[ \mathbf{D}_N \mathbf{u}^N = \mathbf{S}_N \mathbf{f}^{N,ex}, \]
where $\mathbf{f}^{N,ex}$ is a discretization of the source term on a grid of stepsize $h=(N+1)^{-1}$, $ \mathbf{D}_N$ and $ \mathbf{S}_N$ are matrices and $\mathbf{u}^{N}$ is the approximation of the solution of the Dirichlet problem.

\medskip

To study their convergence (i.e. the approximation of the exact solution by $\mathbf{u}^N$), we first introduce some specific and rigorous notions of consistency and stability taking into account the boundary conditions. Then,  we describe precisely the class of schemes that we consider, namely when the matrix  $\mathbf{D}_N$ is a polynomial in the usual discrete second derivative matrix $\mathbf{A}_N$ defined by 
\begin{equation}
\label{definition_du_prince}
\mathbf{A}_N = \left(  \begin{matrix} 2 & -1 \\
													-1 & 2  & -1 \\
													   & \ddots & \ddots & \ddots \\
													   & & -1 & 2 & -1 \\
													   & & & -1 & 2
\end{matrix} \right) \in \mathscr{L}(\mathbb{C}^N). 
\end{equation}
This choice is made for two reasons. First it allows for a relatively simple stability analysis, and second it is in fact not so restrictive. Indeed, if we take a symmetric finite difference formula $d=(d_j)_{j\in \mathbb{Z}}$ that approximates the second derivative, i.e.\ for all smooth function $u$, 
\[ \sum_{j \in \mathbb{Z}} d_{j} u(hj) \simeq -h^2 u''(0), \]
then we get, from a convolution formula and for some specific and natural choice of the coefficients near the boundary, a matrix $\mathbf{D}_N$ that is a polynomial in $\mathbf{A}_N$. 

\medskip

In this paper, we give some general criterion of convergence for this family of schemes. Moreover, we address the following two questions: \begin{itemize}
\item Are these schemes stable {\em in general}\, ?
\item Amongst these schemes, what are the most {\em efficient} and are they stable? 
\end{itemize}
We will precise the two ambiguous terms {\em general} and {\em efficient} by introducing, on one hand, a Lebesgue measure on the set of schemes, and on the other hand, an optimization problem defining efficiency. The first main result of this paper will be to prove that almost all schemes are convergent at the same rate as its consistency order, up to some logarithmic correction. It is based on a careful analysis of small denominators appearing in the stability conditions, linked with diophantine approximation theory. The second main result of this paper is the design and construction of the most efficient schemes in the class considered, which turn to be stable, this latter property requiring the use of Pad\'e approximant theory to be proved. 

\section{Formalism and main results}
The goal of this section is to present the two main results of this paper. To this aim, we first define rigorously compact finite difference schemes for the homogeneous Dirichlet problem in dimension $1$. Then, we recall the usual concept of convergence, consistency and stability for these schemes. And finally, we define the particular set of schemes that we consider. 
\subsection{Context}

We consider the homogeneous Dirichlet problem in dimension $1$, namely: 
\medskip

For a given $f:\mathbb{R}\to \mathbb{C}$,  find $u:[0,1]\to \mathbb{C}$ such that
\begin{equation}
	\label{Dirichet_homogene}
	\left\{ \begin{array}{lll}
 		-u''(x)=f(x), \ \forall x\in]0,1[, \\
 		u(0)=u(1)=0.
	\end{array}\right.
\end{equation} 

\medskip

To design finite difference schemes, we will consider regular grids on $\mathbb{R}$. More precisely, we choose $N\in \mathbb{N}^*$ to be the number of grid points into $]0,1[$ (the number of unknowns) and we define $h$ as the stepsize of the grid. As a consequence, $h$ and $N$ are linked by the relation
\[ h=\frac1{N+1} .\]
Let $x_j^N = jh$, $j \in \mathbb{Z}$ denote the grid points, see Figure \ref{regualr_grid}.
\begin{figure}[h!]

\setlength{\unitlength}{1cm}
\begin{picture}(15,2)
   \put(0,0.5){\line(1,0){7.5}}
   \put(9.5,0.5){\line(1,0){5.5}}
   
   \put(1,0.25){\line(0,1){0.5}}
   \put(3,0){\line(0,1){1}}
   \put(5,0.25){\line(0,1){0.5}}
   \put(7,0.25){\line(0,1){0.5}}

    \put(10,0.25){\line(0,1){0.5}}
	\put(12,0){\line(0,1){1}}
	\put(14,0.25){\line(0,1){0.5}}
	
   \put(8.4,0.5){$\dots$}

	\put(0.3,1.2){$x_{-1}^N=-h$}
	\put(2.5,1.2){$x_{0}^N=0$}
	\put(4.5,1.2){$x_{1}^N=h$}
	\put(6.4,1.2){$x_{2}^N=2h$}
	
	\put(9.4,1.2){$x_{N}^N=Nh$}
	\put(11.3,1.2){$x_{N+1}^N=1$}
	\put(13.1,1.2){$x_{N+2}^N=1+h$}
	
\end{picture}
\caption{\label{regualr_grid} Regular grid with $N$ points into $]0,1[$.}
\end{figure}
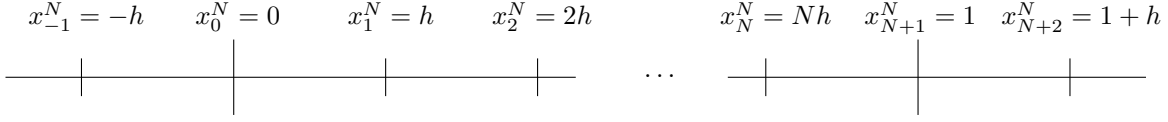

\medskip

In this context, a finite difference scheme is a couple of sequences of matrices $(( \mathbf{D}_N)_{N\in \mathbb{N}^*},( \mathbf{S}_N)_{N \in \mathbb{N}^*})$ such that $ \mathbf{D}_N\in  \mathscr{L}(\mathbb{C}^N)$ is a square matrix of size $N$ and $ \mathbf{S}_N \in \mathscr{L}(\mathbb{C}^\mathbb{Z};\mathbb{C}^N)$ is a rectangle matrix with $N$ rows and a finite number of columns.

\medskip

If $ \mathbf{D}_N$ is invertible for all $N$, such a scheme leads to an approximation of the solution $u$ of the Dirichlet problem \eqref{Dirichet_homogene}. More precisely, we define $\mathbf{f}^{N,ex}$ and $\mathbf{u}^{N,ex}$ as the vectors of the values of $f$ and $u$ on the grid:
\begin{equation}
\label{discretization}
\mathbf{f}^{N,ex}=(f(x_j^N))_{j\in \mathbb{Z}}\in \mathbb{C}^{\mathbb{Z}} \textrm{ and }  \mathbf{u}^{N,ex}=(u(x_j^N))_{j\in \llbracket 1, N \rrbracket }\in \mathbb{C}^N.
\end{equation}
Then, the scheme gives an approximation $\mathbf{u}^{N}$ of $\mathbf{u}^{N,ex}$ through the solution of the linear system
\begin{equation}
\label{linear system}
 \mathbf{D}_N \mathbf{u}^{N} = h^2  \mathbf{S}_N  \mathbf{f}^{N,ex}. 
\end{equation}

\medskip

It may seem unusual to use the values of $f$ outside $[0,1]$ but it is just a way to have more symmetric formulas. In practice, as we will explain in the subsection \ref{Form of the schemes}, we only use a finite number of values of $f$ outside $[0,1]$ independent of $N$ and at a distance of order $h$ of $[0,1]$. Consequently, as we will assume that $f$ is a regular function, these values could be extrapolated from those of $f$ in $[0,1]$ with some Newton series.

\medskip

To estimate the accuracy of a scheme, we define a notion of rate of convergence and of order of convergence. Let $(\epsilon_N)_{N \in \mathbb{N}^*}$ be a sequence of positive numbers that tends to $0$, as $N$ goes to infinity. Then, a scheme $(( \mathbf{D}_N)_{N\in \mathbb{N}^*},( \mathbf{S}_N)_{N \in \mathbb{N}^*})$  is said to be convergent at the rate $(\epsilon_N)_{N \in \mathbb{N}^*}$, if $\mathbf{D}_N$ is invertible for all $N\in \mathbb{N}^*$ and if for all $f\in C^\infty(\mathbb{R})$ there exists a constant $c>0$ such that for all $N\in \mathbb{N}^*$,
\begin{equation}
\label{convergence estimation}
 \sup_{j=1,\dots, N} |\mathbf{u}^{N}_j - \mathbf{u}^{N,ex}_j | =: \| \mathbf{u}^{N,ex}  - \mathbf{u}^{N} \|_{\infty} \leq c \epsilon_N. 
\end{equation}
Furthermore, if $n$ is a positive integer and $\epsilon_N=h^n$, then a scheme that is convergent at the rate $(\epsilon_N)_{N \in \mathbb{N}^*}$ is said to be convergent of order $n$.

\subsection{Notions of consistency and stability}
In order to establish a convergence result of the form \eqref{convergence estimation}, we use introduce the notions of consistency and stability. Then, we give a Lax theorem to deduce the convergence from the consistency and the stability.

\medskip

A scheme $(( \mathbf{D}_N),( \mathbf{S}_N))$ is said to be consistent of order $n\in \mathbb{N}$, if for all $f\in C^{\infty}(\mathbb{R})$ there exists a constant $c>0$ such that for all $N\in \mathbb{N}^*$, the vectors $\mathbf{u}^{N,ex} $ and $\mathbf{f}^{N,ex}$, defined by \eqref{discretization}, verify
\begin{equation}
\label{def consistency}
 \|\mathbf{D}_N \mathbf{u}^{N,ex} - h^2 \mathbf{S}_N \mathbf{f}^{N,ex} \|_{\infty} \leq c h^{n+2}.
\end{equation}

\medskip

In the context of the Dirichlet problem, it is usual to relax this notion of consistency near the boundary (see \cite{MR0165702}). A scheme $(( \mathbf{D}_N),( \mathbf{S}_N))$ is said to be consistent of order $n\in \mathbb{N}$ in the center and of order $n-2$ at a distance $l \in \mathbb{N}$ of the boundary, if for all $f\in C^{\infty}(\mathbb{R})$ there exists a constant $c>0$ such that for all $N\in \mathbb{N}^*$, the vectors $\mathbf{u}^{N,ex} $ and $\mathbf{f}^{N,ex}$, defined by \eqref{discretization}, verify for all $j=1,\dots, N$,
\begin{equation}
\label{def extend consistency}
 \left| \left( \mathbf{D}_N \mathbf{u}^{N,ex}\right)_j - h^2 \left( \mathbf{S}_N \mathbf{f}^{N,ex} \right)_j  \right| \leq \left\{  \begin{array}{lll} c h^{n+2} & \textrm{ if } l<j<N+1-l ,\\
																																																	c h^{n}  & \textrm{ else.}
\end{array} \right.  
\end{equation}

\medskip

In this article, it is useful to distinguish some notions of stability. A scheme $(( \mathbf{D}_N),( \mathbf{S}_N))$ or a sequence $\displaystyle ( \mathbf{D}_N)_{N\in \mathbb{N}^*}\in \prod_{N\in \mathbb{N}^*} \mathscr{L}(\mathbb{C}^N)$ of matrices is said to be
\begin{itemize}
\item stable, if there exists a positive constant $c>0$ such that for all $N\in \mathbb{N}^*$, we have
\begin{equation}
\label{def stability}
\forall \mathbf{v}\in \mathbb{C}^N, \ c \| \mathbf{v} \|_{\infty} \leq h^{-2} \| \mathbf{D}_N  \mathbf{v} \|_{\infty}.
\end{equation}
\item strongly stable,  if for all $l \in \mathbb{N}$, there exists a positive constant $c>0$ such that for all $N\in \mathbb{N}^*$,
\begin{equation}
\label{def strong stability}
\forall \mathbf{v}\in \mathbb{C}^N, \ c \| \mathbf{v} \|_{\infty} \leq \sup_{j=1,\dots,N}   \left\{  \begin{array}{lll}  h^{-2}\left( \mathbf{D}_N  \mathbf{v} \right)_j & \textrm{ if } l<j<N+1-l ,\\
																																																	 \left( \mathbf{D}_N  \mathbf{v} \right)_j  & \textrm{ else.}
\end{array} \right..
\end{equation}
\item stable relatively to a sequence $(\eta_N)_{N\in \mathbb{N}^*}$ of positive numbers, if there exists a positive constant $c>0$ such that for all $N\in \mathbb{N}^*$, we have
\begin{equation}
\label{def relative stability}
\forall \mathbf{v}\in \mathbb{C}^N, \ c \| \mathbf{v} \|_{\infty} \leq \eta_N \| \mathbf{D}_N  \mathbf{v} \|_{\infty}.
\end{equation}
\end{itemize}
We remark that, if a scheme is strongly stable, then it is stable, and, if it is stable, then it is stable relatively to $\eta_N=(N+1)^2=h^{-2}$.

\medskip

To establish  convergence from  consistency and  stability, we give a Lax theorem.
\begin{theo} Lax
\label{Lax}
\begin{itemize}
\item A scheme that is strongly stable (see \eqref{def strong stability}) and consistent of order $n\geq 1$ in the center and of order $n-2$ at a distance $l \in \mathbb{N}$ of the boundary (see \eqref{def extend consistency}) is convergent of order $n$.
\item Let $(\eta_N)_{N\in \mathbb{N}^*}$ be a sequence of positive number and $n\in \mathbb{N}^*$ such that the sequence $(\eta_N h^{n+2})_{N\in \mathbb{N}^*}$ tends to zero as $N$ goes to infinity. Then a scheme that is stable relatively to the sequence $(\eta_N)_{N\in \mathbb{N}^*}$ \eqref{def relative stability} and consistent of order $n$ \eqref{def consistency} is convergent at the rate $\epsilon_N=\eta_N h^{n+2}$ \eqref{convergence estimation}.
\end{itemize}
\end{theo}
\begin{proof}
The invertibility of $\mathbf{D}_N$ follows from the stability estimate. To prove the convergence estimate, it is enough to apply the stability estimate to the error of consistency
\[\mathbf{D}_N \mathbf{v}=\mathbf{D}_N \left( \mathbf{u}^{N,ex} -\mathbf{u}^{N} \right) = \mathbf{D}_N \mathbf{u}^{N,ex} - h^2 \mathbf{S}_N \mathbf{f}^{N,ex} .\]
\end{proof}

\subsection{Expression of the schemes}
\label{Form of the schemes}
Usually, to design a finite difference scheme $(( \mathbf{D}_N),( \mathbf{S}_N))$, we need to introduce the notion of finite difference formulas. A finite difference formula is a sequence of complex numbers indexed by $\mathbb{Z}$ with  finite support. We denote by $\mathbb{C}^{(\mathbb{Z})}$ their space. We say that a couple of finite difference formulas $(d,s) \in \left( \mathbb{C}^{(\mathbb{Z})} \right)^2$ is consistent of order $n$, if
\begin{equation}
\label{def consistency formula}
\forall u \in C^{\infty}(\mathbb{R}), \ \sum_{j\in \mathbb{Z}} d_j u(x^N_j) + h^2 s_j u''(x^N_j) =  \mathcal{O}(h^{n+2}).
\end{equation}
For example, if we introduce the usual formula for the second derivative
\begin{equation}
\label{definition_formule_du_prince}
a = 2\mathbb{1}_{\{0\}}-\mathbb{1}_{\{-1,1\}},
\end{equation}
then a Taylor expansion shows that $(a,\mathbb{1}_{\{0 \}})$ is consistent of order $2$.

\medskip

To preserve the classical properties of the second derivative, it is natural to assume that the sequences $d$ and $s$ are symmetric,\begin{equation}
\label{def symetric complex}
d,s\in \mathscr{S}_{\mathbb{C}}:=\{ b \in  \mathbb{C}^{(\mathbb{Z})} \ | \ \forall j\in \mathbb{Z}, \ b_j=b_{-j}\}, 
\end{equation}
and it is then natural to restrict the analysis to the case where $n$ is an even number. 
Sometimes, it is interesting and more effective --for instance using formal calculus-- to consider finite difference formulas with coefficients in a smaller ring than $\mathbb{C}$. For example, the usual high order formulas have rational or integer coefficients. That is why, we introduce, the more general notation
\begin{equation}
\label{def symetric corps}
\mathscr{S}_{R}:=\{ b \in R^{(\mathbb{Z})} \ | \ \forall j\in \mathbb{Z}, \ b_j=b_{-j}\} \textrm{ with } R \textrm{ a ring such that } \mathbb{Z}\subset R \subset \mathbb{C}.
\end{equation}
It is useful to associate to each finite difference formula the highest index associated to a non zero value. It is a measure of the stencil of a formula. More formally, if $b\in \mathscr{S}_{\mathbb{C}}$ is a symmetric formula
then $\tau(b)$ is defined by
\begin{equation}
\label{def size stencil}
\tau(b)= \max \{ j \in \mathbb{Z} \ | \ b_j\neq 0 \}.
\end{equation}
The following proposition explains that there is a simple way to get finite difference formulas $d,s\in \mathscr{S}_{\mathbb{C}}$ consistent of order $n$. 
\begin{propo} Let $n\in 2\mathbb{N}$ be an even integer and $d\in \mathscr{S}_{\mathbb{C}}$ be a symmetric formula with zero mean 
\label{prop si on a d alors on a s}
\begin{equation}
\label{zero mean}
\sum_{j\in \mathbb{Z}} d_j=0.
\end{equation}
Then there exists a unique $s\in \mathscr{S}_{\mathbb{C}}$ such that $(d,s)$ is consistent of order $n$ \eqref{def consistency formula} and $\tau(s)\leq \frac{n}2-1.$
Furthermore, $(\frac{s_0}2, s_1,\dots,s_{\frac{n}2-1}) $ is the solution of the Vandermonde linear system
\begin{equation}
\label{Vandermonde}
 ( \frac{s_0}2, s_1,\dots,s_{\frac{n}2-1})((i-1)^{2j-2})_{1\leq i,j \leq \frac{n}2} = -  \sum_{j >0 } d_j \left( \frac{j^2}2,\dots,\frac{j^n}{n(n-1)} \right)   .
\end{equation}
\end{propo}
\begin{proof}
If $1\leq j\leq \frac{n}2$ is an integer and if we choose $u=x^{2j}$ in \eqref{def consistency formula} then it comes 
\[  \sum_{i\in \mathbb{Z}} d_i (hi)^{2j} + s_j 2j(2j-1) h^{2j} i^{2(j-1)}   =  \mathcal{O}(h^{n+2}). \]
As $j\leq \frac{n}2$ and $h$ tends to $0$, we deduce that the remainder vanishes and we recognize the Vandermonde equation \eqref{Vandermonde}.\\
Conversely, since $d$ and $s$ are symmetric, if $u$ is an odd function then
\[ \sum_{i\in \mathbb{Z}} d_i u(x^N_i) + h^2 s_i u''(x^N_i) =0. \]
Furthermore, since $s$ is the solution of \eqref{Vandermonde}, this relation also holds if $u=x^{2j}$ with $1\leq j\leq \frac{n}2$. As a consequence, it is enough to apply a Taylor Young expansion to prove \eqref{def consistency formula}.
\end{proof}

\medskip

Then, to design the matrix $\mathbf{D}_N$ and $\mathbf{S}_N$ from the formulas $d$ and $s$, a natural choice would be the following: 
\begin{equation}
\label{def a l arrache}
\left(\mathbf{D}_N\mathbf{u}\right)_i=\sum_{j\in \mathbb{Z}} d_{i-j} \mathbf{u}_{j} \textrm{ and }  \left(\mathbf{S}_N\mathbf{f}\right)_i=\sum_{j\in \mathbb{Z}} s_{i-j} \mathbf{f}_{j} .
\end{equation}
However, $\mathbf{D}_N$ has to be square matrix. And, with such a definition, we use the values of $\mathbf{u}$ at the indexes $1-\tau(d),\dots,0$ and $N+1,\dots,N+\tau(d)$.
The usual way to solve this problem is to modify the formulas near the boundary (for $i\leq\tau(d)$ or $i\geq N+1-\tau(d)$).  That is why, we introduce, for $i=1,\dots,\tau(d)$, some formulas $d^i\in \mathbb{C}^{(\mathbb{Z})}$ and $s^i\in \mathbb{C}^{(\mathbb{Z})}$ that satisfy a relation of consistency at a distance $i$ of the boundary
\begin{equation}
\label{def consistency formula boundary}
\forall u \in C^{\infty}(\mathbb{R}), \ u(0)=0 \ \Rightarrow \ \sum_{j > -i} d_j^i u(x^N_{j+i}) + h^2 \sum_{j\in \mathbb{Z}} s_j^i u''(x^N_{j+i}) =  \mathcal{O}(h^{\mu+2}),
\end{equation}
here $\mu\in \{n-2,n\}$ is the desired order of consistency. We use symmetrically in $1$ these formulas to define, if $N$ is large enough, the following scheme 
$(( \mathbf{D}_N),( \mathbf{S}_N))$,  for $\mathbf{u}\in \mathbb{C}^N$ and $\mathbf{f}\in \mathbb{C}^\mathbb{Z}$, by
\begin{equation}
\label{def abstraite matrices D}
 \left(\mathbf{D}_N\mathbf{u}\right)_i :=\left\{ \begin{array}{llll} \displaystyle \sum_{j>0} d_{j-i}^i \mathbf{u}_{j} & \textrm{ if } 1\leq i\leq \tau(d),  \\
 																						\displaystyle \sum_{j\in \mathbb{Z}} d_{j-i} \mathbf{u}_{j} & \textrm{ if } \tau(d)< i < N+1-\tau(d),  \\
 																						\displaystyle \sum_{j<N+1} d_{-j+i}^{N+1-i} \mathbf{u}_{j} & \textrm{ if } N+1-\tau(d) \leq i \leq N+1.
\end{array}  \right. 
\end{equation}
and
\begin{equation}
\label{def abstraite matrices S}
 \left(\mathbf{S}_N\mathbf{f}\right)_i :=\left\{ \begin{array}{llll} \displaystyle \sum_{j\in \mathbb{Z}} s_{j-i}^i \mathbf{f}_{j} & \textrm{ if } 1\leq i\leq \tau(d),  \\
 																						\displaystyle \sum_{j\in \mathbb{Z}} s_{j-i} \mathbf{f}_{j} & \textrm{ if }  \tau(d)< i < N+1-\tau(d),  \\
 																						\displaystyle \sum_{j\in \mathbb{Z}} s_{-j+i}^{N+1-i} \mathbf{f}_{j} & \textrm{ if } N+1-\tau(d) \leq i \leq N+1.  \\
\end{array}  \right. 
\end{equation}
The following proposition enables to get the consistency of such a construction.
\begin{propo} 
\label{prop the schemes are consistent}
For $N$ large enough, let $(( \mathbf{D}_N),( \mathbf{S}_N))$ be the scheme (defined by  \eqref{def abstraite matrices D} and \eqref{def abstraite matrices S}), then
\begin{itemize}
\item if $\mu=n-2$, this scheme is consistent of order $n-2$ at a distance $\tau(d)$ of the boundary and of order $n$ in the center, see \eqref{def extend consistency}.
\item if $\mu=n$, this scheme is consistent of order $n$, see \eqref{def consistency}.
\end{itemize}
\end{propo}
\begin{proof}
see Appendix \ref{proof the schemes are consistent}.
\end{proof}
The main difficulty with such a construction is to get stability. There are at least two general ways for choosing the formulas near the boundary to ensure stability. A first principle is to rely on  monotonicity arguments, as explained by  Bramble and Hubbart  \cite{MR0165702} and Price \cite{MR0232550}. The methods they consider to design the coefficients near the boundary are robust and lead, in general, to strong stability. However, the choice of formulas $d$ and $d^i$ is quite limited, as the conditions to ensure monotonicity are in general difficult to fulfil. Furthermore, it turns out that there exist very accurate high order schemes that do not satisfy any hypothesis of monotonicity. 

A second natural way of obtaining the boundary coefficients is to start from {\em polynomial methods} that we consider below. For these methods, if we respect some algebraic structures, we can compute explicitly the eigenvalues and the eigenvectors of $\mathbf{D}_N$, and analyse directly the stability. This method is not very restrictive for the choice of the formulas $d$ and there is a natural choice for the formulas $d^i$ near the boundary.

\medskip
The polynomial methods consists in studying schemes for which there exists a polynomial $P$ such that, for all $N\in \mathbb{N}$, $\mathbf{D}_N = P(\mathbf{A}_N)$ is a polynomial of $\mathbf{A}_N$ (the classical approximation of the second derivative, defined in \eqref{definition_du_prince}). The interest of this method is that the spectral decomposition of these matrices is well known. Indeed, we can verify by a straightforward calculation that
\begin{equation}
	\label{spectral_decompostion}
	\mathbf{A}_N \mathbf{e}^N_k = 4\sin^2 \left( \frac{\pi}2 k h \right) \mathbf{e}^N_k,  \ \textrm{ with } \mathbf{e}^N_k := (\sin(\pi k h j))_{j=1,\dots,N},
\end{equation}
and deduce classically that
\begin{equation}
\label{scarlatti}
 \mathbf{D}_N \mathbf{e}^N_k = P\left( 4\sin^2 \left( \frac{\pi}2 k h \right) \right) \mathbf{e}^N_k.  
 \end{equation}

\medskip

Actually, it is not very restrictive to require for $\mathbf{D}_N$ to be a polynomial in $\mathbf{A}_N$. 
Indeed, for a given symmetric formulas $d$, there is a natural possible choice for the boundary formulas $d^i$, $i = 1,\ldots,\tau(d)$ such that the matrix $\mathbf{D}_N$ defined by \eqref{def abstraite matrices D} is a polynomial in  $\mathbf{A}_N$. This choice corresponds to extend all the vectors $\mathbf{u}\in \mathbb{C}^N$ in sequences defined on $\mathbb{Z}$ through the relations
\[ \forall j \in \mathbb{Z}, \ \mathbf{u}_j = - \mathbf{u}_{-j} \textrm{ and } \mathbf{u}_{N+1+j} =- \mathbf{u}_{N+1-j},\]
and use the natural convolution formula \eqref{def a l arrache}. 
In practice, when $N$ is large enough, this choice leads to 
\begin{equation}
\label{chopin}
 d^i_j = d_j - d_{2i+j},  \quad i = 1,\ldots,\tau(d), \quad j \in \mathbb{Z}
\end{equation}
In all this paper, we denote by $\mathbf{D}_N(d)$ the square matrix obtained from this construction (i.e.\ the matrix \eqref{def abstraite matrices D} and the boundary formulas \eqref{chopin}-- see Definition \ref{def formal D_N}  for a formal construction).

\medskip

The following proposition shows that the previous construction is relevant:  First, we prove that all the matrices $\mathbf{D}_N(d)$ are polynomials in $\mathbf{A}_N$, and second we can find formulas $s^i$, $i = 1,\ldots,\tau(d)$ satisfying \eqref{def consistency formula boundary} for any given order of consistency $\mu$.
\begin{propo}
\label{prop de D_N} $\empty$
\begin{itemize}
\item If $R$ is a ring such that $\mathbb{Z}\subset R \subset \mathbb{C}$ and if $d\in \mathscr{S}_R$ is a $R$ valued finite difference symmetric formula then there exists a polynomial $P\in R[X]$ such that
\[ \forall N\in \mathbb{N}^*, \ P(\mathbf{A}_N) = \mathbf{D}_N(d) \]
and
\[ \deg P = \tau(d) .\]
\item Let $n\in 2\mathbb{N}^*$ and $\mu=n$ or $\mu=n-2$. If there exists a finite difference formula $s\in \mathscr{S}_{\mathbb{C}}$ such that $(d,s)$ is consistent of order $\mu$ (see \eqref{def consistency formula}) then for all $i=1,\dots,\tau(d)$ there exists a unique symmetric formula $b^i\in \mathscr{S}_{\mathbb{C}}$ such that $\tau(b)\leq \frac{\mu}2-1$ and 
\[ s^i :=  s + (b_{i+j}^i)_{j\in \mathbb{Z}} \textrm{ is consistent of order } \mu \textrm{ at a distance } i \textrm{ of the boundary, see \eqref{def consistency formula boundary}}.\] 
Furthermore, $( \frac{b_0^i}2, b_1^i,\dots,b_{\frac{\mu}2-1}^i) $ is the solution of the Vandermonde linear system
\begin{equation}
\label{Vandermonde boundary}
 (\frac{b_0^i}2, b_1^i,\dots,b_{\frac{\mu}2-1}^i)((i-1)^{2j-2})_{1\leq i,j \leq \frac{\mu}2} = - \sum_{j > 0} d_{i+j} \left( \frac{j^2}2,\dots,\frac{j^\mu}{\mu(\mu-1)} \right)   .
\end{equation}
\end{itemize}
\end{propo}
\begin{proof}
The first point will be proved in the next section as a direct consequence of Lemma \ref{lemme toute formule est un polynome du prince} and Lemma \ref{lemme D_N est un morphisme de module}.
To prove the second point, let consider $u \in C^{\infty}(\mathbb{R})$ such that $u(0)=0$. Then we have from \eqref{chopin}, for $i = 1,\ldots, \tau(d)$, 
\begin{align*}
\sum_{j > -i} d_j^i u(x^N_{j+i}) + h^2 \sum_{j\in \mathbb{Z}} s_j u''(x^N_{j+i}) &= \sum_{j > -i} \left( d_j - d_{j+2i} \right) u(x^N_{j+i}) + h^2 \sum_{j\in \mathbb{Z}} s_j u''(x^N_{j+i}) \\
																								   &= \sum_{j \in \mathbb{Z}}  d_j u(x^N_{j+i}) + h^2 s_j u''(x^N_{j+i}) - \sum_{j< -i } d_j u(x^N_{j+i}) - \sum_{j > -i } d_{j+2i} u(x^N_{j+i}) \\
																								   &= - \sum_{ j>0} d_{i+j} \left( u(x^N_{-j})+ u(x^N_{j}) \right) + \mathcal{O}(h^{\mu+2})\\
																								   &= - \sum_{ j\in \mathbb{Z}} \widetilde{d_j} u(x^N_{j}) + \mathcal{O}(h^{\mu+2}),
\end{align*}  
with $ \widetilde{d} \in \mathscr{S}_{\mathbb{C}} $ a symmetric finite difference formula with zero mean \eqref{zero mean} defined by $\widetilde{d_j} = d_{i+j} $ if $j>0$. Then applying Proposition \ref{prop si on a d alors on a s} enables to conclude the proof.
\end{proof}

\begin{rem}
The formula \eqref{Vandermonde boundary} implies in particular that $b^{\tau(d)}=0$ because, for $i=1,\dots,\tau(d)$, the right hand side term in \eqref{Vandermonde boundary} is zero by definition of $\tau(d)$.
\end{rem}
\medskip

To conclude this part, explicit expressions of a class a high order schemes constructed using the previous principle are proposed. They will be used to give examples.
\begin{propo}
\label{le schema de base}
Let $(d,s)\in \mathscr{S}_{\mathbb{C}}$ be  a couple of symmetric finite difference formulas that is consistent of order $n$ with $n\in 2\mathbb{N}^*$. Let $\mu\in \{n-2,n\}$ be an even integer. Define $l=\tau(d)-1$, $m=\tau(s)$ and for $i=1,\dots,l$, $b^i$ as the solution of the system \eqref{Vandermonde boundary}. If we choose $s^i=s+(b_{i+j}^i)_{j\in \mathbb{Z}} $ and $d^i=d_j - d_{2i+j}$ then the relations \eqref{def abstraite matrices D} and \eqref{def abstraite matrices S} define a scheme that is consistent of order $n$ in the center and of order $\mu$ at a distance $l$ of the boundary. More precisely, if $N$ is large enough, this scheme is given by the following band matrices:
\[ 
\mathbf{D}_N(d) = \left( \begin{matrix} d_0 & \dots & d_{l+1} \\
													\vdots  & \ddots &  & \ddots\\
													 d_{l+1} &  & \ddots &  &  \ddots \\
 															& \ddots && \ddots && d_{l+1}\\
 															&& \ddots && \ddots & \vdots \\
 															&&& d_{l+1} & \dots & d_0
\end{matrix} \right) - \left( \begin{matrix} d_2 & \dots & d_{l+1} \\
													\vdots  & \mathrm{\reflectbox{$\ddots$}} \\
													 d_{l+1} \\
 															&  && && d_{l+1}\\
 															&&  && \mathrm{\reflectbox{$\ddots$}} & \vdots \\
 															&&& d_{l+1} & \dots & d_2
\end{matrix} \right)
\in \mathscr{L}(\mathbb{C}^N),  \]
 \[ 
\mathbf{S}_N = \left( \begin{matrix} s_{m} & \dots & s_0 & \dots & s_m \\
																& \ddots &   & \ddots &  & \ddots \\
																&& \ddots &   & \ddots &  & \ddots \\
																&&& \ddots &   & \ddots &  & \ddots \\
																&&&& \ddots &   & \ddots &  & \ddots \\
																&&&&& s_{m} & \dots & s_0 & \dots & s_m \\
\end{matrix} \right) + \left( \begin{matrix} \mathbf{B}_{\mu}^{+} \\
														  &  \mathbf{0}_{N-\mu+2,N-\mu+2}  \\
														  & &  \mathbf{B}_{\mu}^{-}
\end{matrix} \right)
 \in \mathscr{L}(\mathbb{C}^\mathbb{Z};\mathbb{C}^N)\] 
 with  $\mathbf{0}_{N-\mu+2,N-\mu+2}$ the zero square matrix of size $N-\mu+2$,
 \[ 
\mathbf{B}_{\mu}^{+} = \left(  \begin{matrix} b^1_{ \frac{\mu}2-1 } & \dots  & \dots & b^1_{ 0 }  & \dots &  \dots & b^1_{ \frac{\mu}2-1 } \\
															 \vdots &  \vdots &  \vdots & \vdots &  \vdots &  \vdots & \vdots \\
															  b^{l}_{ \frac{\mu}2-1 }  & \dots & \dots & b^{l}_{ 0 }  & \dots & \dots &   b^{l}_{ \frac{\mu}2-1 }
\end{matrix} \right) \in  \mathscr{L}(\mathbb{C}^{\mu -1};\mathbb{C}^{l}) \]
and
\[ \mathbf{B}_{\mu}^{-} = \left(  \begin{matrix} b^{l}_{ \frac{\mu}2-1 }  & \dots & \dots & b^{l}_{ 0 }  & \dots & \dots &   b^{l}_{ \frac{\mu}2-1 }\\
															 \vdots &  \vdots &  \vdots & \vdots &  \vdots &  \vdots & \vdots \\
															 b^1_{ \frac{\mu}2-1 } & \dots  & \dots & b^1_{ 0 }  & \dots &  \dots & b^1_{ \frac{\mu}2-1 } 												  
\end{matrix} \right) \in \mathscr{L}(\mathbb{C}^{\mu -1};\mathbb{C}^{l}).
 \]
 \end{propo}
 
\subsection{Main results}
%

We will first define the notion of {\em efficiency} discussed in the introduction. 
If we design our schemes as in Proposition \ref{le schema de base}, and unless some more specific algebraic structure is given, the computation time for the approximation of one solution of the Dirichlet problem -- that is the solution of the linear system \eqref{linear system}--  grows {\em a priori} linearly with the size of the stencils $\tau(d)$ and $\tau(s)$.  As a consequence, in general, the smaller $\tau(d)$ and $\tau(s)$ are, the larger the order of consistency of $(d,s)$ is, and hence the more \it efficient \rm is our scheme. As a consequence, we will define schemes to be {\em the most efficient} for given $l,m\in \mathbb{N}$, those which are solutions to the following optimization problem:  \begin{equation}
\label{le pb dopti}
 \mathop{\max_{(d,s)\in \mathscr{S}_{\mathbb{C}}^2\setminus \{(0,0) \}}}_{\tau(d)\leq l+1,\ \tau(s)\leq m} \ord(d,s),
\end{equation}
where $\ord(d,s)$ is the exact order of consistency of $(d,s)$
\[ \ord(d,s)=\sup \{n\in 2\mathbb{N} \ | \ (d,s) \textrm{ is consistent of order } n \textrm{ according to \eqref{def consistency formula}}\}. \]

\medskip

The following theorem proves that for any given stencil sizes $l$ and $m$ in $\mathbb{N}$, there exists a most efficient scheme solution of the previous optimization problem, and it is unique, up to a multiplication by a scalar. 
\begin{theo} \label{theo les formules optimales}
For all $l,m\in \mathbb{N}$, there exists a couple of rational symmetric formulas $(d^{l,m},s^{l,m})\in \mathscr{S}_{\mathbb{Q}}^2$ such that
\[ \left\{  \begin{array}{llll} \tau(d^{l,m})=l+1, \\
									\tau(s^{l,m})=m,\\
									\displaystyle \sum_{j\in \mathbb{Z}} d^{l,m}_j j^2 = -2,
\end{array} \right. \]
 that is solution of the problem of optimization
\[ \mathop{\max_{(d,s)\in \mathscr{S}_{\mathbb{C}}^2\setminus \{(0,0) \}}}_{\tau(d)\leq l+1,\ \tau(s)\leq m} \ord(d,s) = \ord(d^{l,m},s^{l,m})= 2(l+m+1).\]
Moreover if $(d,s)\in \mathscr{S}_{\mathbb{C}}^2$ is such that $\tau(d)\leq l+1$, $\tau(s)\leq m$ and $\ord(d,s)=2(l+m+1)$ then there exists $\lambda\in \mathbb{C}$ such that $d=\lambda d^{l,m}$ and $s=\lambda s^{l,m}$.
\end{theo}
This theorem is the main result of the second section of this work (see Theorem \ref{theo les poly optimaux}). The proof relies on an interpretation of the optimization problem \eqref{le pb dopti} as Pad\'e approximant problem. The optimal formulas of this theorem are effective because we can prove, with the property of uniqueness of Theorem \ref{theo les schemas optimaux sont stables}, that they can be computed exactly as the solutions of these rational $(l + m + 3) \times (l + m + 3)$ linear systems
\begin{equation}
\label{linear system for optimal formulas}
\left( \begin{matrix} \mathbf{L}_{l+1}^{0}   & \mathbf{0}_{1,m+1}        \\
							  \mathbf{L}_{l+1}^{2}   & \mathbf{0}_{1,m+1}        \\
							  \mathbf{L}_{l+1}^{2} & 2\mathbf{L}_{m}^{0} \\
							 \vdots & \vdots \\
							    \mathbf{L}_{l+1}^{2(l+m+1)}  & 2(m+l+1)(2(m+l)+1)\mathbf{L}_{m}^{2(l+m)}
\end{matrix} \right) \left( \begin{matrix}  d_0^{l,m} \\  \vdots \\  d_{l+1}^{l,m}\\  s_0^{l,m} \\ \vdots \\  s_{m}^{l,m} \end{matrix}  \right) = \left( \begin{matrix} 0 \\ -1 \\ 0 \\ \vdots \\ 0 \end{matrix}  \right) ,
\end{equation}
with 
\[ \mathbf{L}_{n}^{k} =  \left( \begin{matrix}  \frac{0^k}2 & 1^k & \dots & n^k  \end{matrix} \right) \textrm{ and } \mathbf{0}_{1,m+1} \textrm{ the zero row matrix of size }m+1. \]

\medskip

The formulas $d^{l,m}$ being constructed as the solution of an optimization problem, there is {\em a priori} no reason that they generate stable schemes. 
However, the following theorem, which will be proved in the third section (see Application 2 of Theorem \ref{theorem_positive_case}), precisely states that  all these optimal schemes are indeed strongly stable. 
\begin{theo} \label{theo les schemas optimaux sont stables}
For all $l\in \mathbb{N}^*$ and for all $m\in \mathbb{N}$, $(\mathbf{D}_N(d^{l,m}))_{N\in \mathbb{N}^*}$ is strongly stable, see \eqref{def strong stability}.
\end{theo}
In particular, following Theorem \ref{Lax}, the schemes designed in Proposition \ref{le schema de base}, with $d=d^{l,m}$, $s={l,m}$, $n=2(l+m+1)$ and $\mu=2(l+m)$, are convergent of order $2(l+m+1)$.\\
Experimentally, these schemes are very efficient and we can notice that the smaller $|m-l|$ is, the more accurate the scheme is.
This is illustrated in Figure \ref{convergence_curves_ordre_10} in which some convergence plots are displayed for the $10^{th}$ order optimal formulas (i.e. $l + m + 1 = 5$)

\medskip

\begin {figure}[!ht]
	\begin{center}
		\includegraphics[width=12cm]{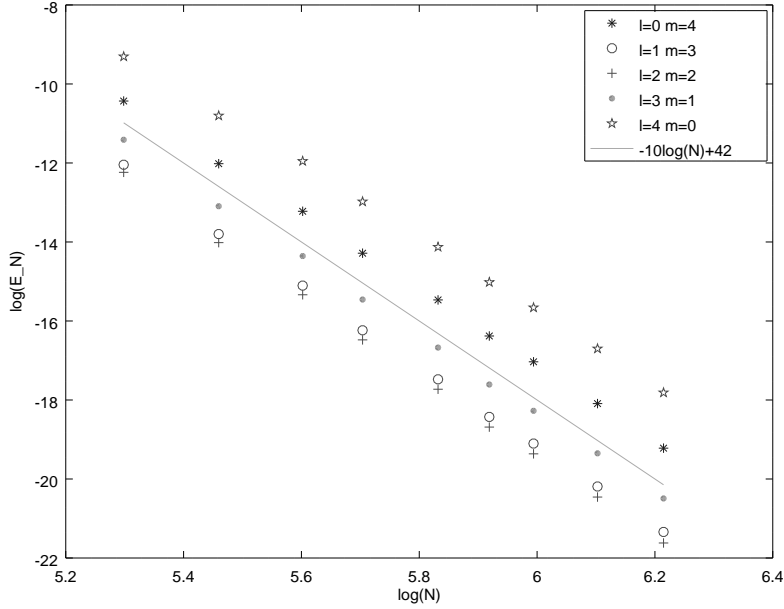}
	\end{center}
	\caption{\label{convergence_curves_ordre_10} Convergence curves, with $u(x)=x(1-x)e^{4\cos(41x)}$  and $E_N:=\|\mathbf{u}^N-\mathbf{u}^{N,ex}\|_{\infty}$, $N\in \{ 200,235,271,300,341,372,401,447,500 \}$, for the optimal schemes designed in Proposition \ref{le schema de base}, with $n=10$, $\mu=8$, $d=d^{l,m}$ and $s=s^{l,m}$.}
\end {figure}

\medskip

As explained in the introduction, we now address the question of {\em generic} performance of the schemes that we have constructed above: are they stable and convergent {\em in general} once the algebraic order conditions are satisfied. 
To give a meaning to this question, we decide to use measure theory. Of course there exist formulas such that $(\mathbf{D}_N(d))_{N}$ can not be stable. It is the case, for example, when $\mathbf{D}_N(d)$ is not invertible for all $N$ which occurs for when the polynomial $P$ defining the scheme admit a root of the form $ 4\sin^2 \left( \frac{\pi}2 k h \right) $, see \eqref{scarlatti}, which are eigenvalues of the matrix $\mathbf{A}_N$. But even if this is not the case, these eigenvalues can be very close to the roots of $P$, which induces small denominators in the stability estimates. Of course, these situations have to be avoided as well. 

The following theorem gives an answer to these questions (see Application 1 of Theorem \ref{theorem_positive_case} and Application 2 of Theorem \ref{le_theoreme_diophantien} for the proof).
\begin{theo}
\label{theo alea}
Let $\mathbb{K}\in \{ \mathbb{R},\mathbb{C}\}$ be a field and $l\in \mathbb{N}$ be an integer. Let $\mathscr{C}_{\mathbb{K},l}$ be the $\mathbb{K}$ finite dimensional vector space of symmetric formulas $d\in \mathscr{S}_{\mathbb{K}}$ with zero mean \eqref{zero mean} and $\tau(d)\leq l+1$ 
\[  \mathscr{C}_{\mathbb{K},l} = \{ d\in \mathscr{S}_{\mathbb{K}} \ | \  \sum_{j\in \mathbb{Z}} d_j= 0 \textrm{ and } \tau(d)\leq l+1 \}. \]
Then for any Lebesgue measure on $\mathscr{C}_{\mathbb{K},l}$, we have that: 
\begin{itemize}
\item For almost all $d\in  \mathscr{C}_{\mathbb{C},l}$, $(\mathbf{D}_N(d))_{N\in \mathbb{N}^*}$ is strongly stable \eqref{def strong stability}.
\item For almost all $d\in  \mathscr{C}_{\mathbb{R},l}$, $(\mathbf{D}_N(d))_{N\in \mathbb{N}^*}$ is stable relatively to any sequence $(\eta_N)_N$ \eqref{def relative stability} such that
\[ \sum_{N\in \mathbb{N}^*} \frac{N+1}{\eta_N} < \infty \textrm{ and }  \sup_{N\in \mathbb{N}^*} \frac{(N+1)^2}{\eta_N} < \infty.  \]
\end{itemize}
\end{theo}
As for Bertrand series, there is no optimal choice of sequence $(\eta_N)_{N\in \mathbb{N}}$ that satisfy this condition and we can not directly  deduce stability in the sense of \eqref{def stability}, but we can choose
\[ \eta_N=((N+1)\log(N+1))^2= \left( \frac{\log h}h \right)^2.\]
As a consequence, we affirm that, up to some logarithmic corrections, almost all real symmetric formula generates stable schemes.

\medskip

We use this theorem to deduce a convergence result.  
\begin{propo}
\label{prop cvg alea}
With any given $d\in \mathscr{C}_{\mathbb{K},l}$, we can associate the scheme of Proposition \ref{le schema de base}, with $\mu=n-2$ if $\mathbb{K}=\mathbb{C}$ and $\mu=n$ if $\mathbb{K}=\mathbb{R}$, and the formula $s$ given by Proposition \ref{prop si on a d alors on a s}. Then, it follows from Theorem \ref{Lax} that for all $l\in \mathbb{N}$,
\begin{itemize}
\item for almost all $d\in  \mathscr{C}_{\mathbb{C},l}$, the associated scheme is convergent of order $n$.
\item for almost all $d\in  \mathscr{C}_{\mathbb{R},l}$, the associated scheme converges at the rate $h^n (\log(h))^2$.
\end{itemize}
\end{propo}

\medskip

In the proof of Theorem \ref{theo alea} for real formulas, the logarithmic correction is due to the use of a diophantine control of some resonances. Experimentally, we can indeed evidence these quasi-resonances by plotting convergence curves for various schemes of Proposition \ref{prop cvg alea} for randomly drawn formulas $d$. Two typical kinds of behaviors for the convergence curves can be observed (see Figures \ref{convergence_curves_ordre_2} and \ref{convergence_curves_ordre_4} below). 
For random $d$, either we observe classical convergence curves (which are close to straight lines and correspond to non-resonant situations), or we obtain strange curves with a complex behaviour corresponding to close to resonant situations. 
\begin {figure}[!ht]
	\begin{center}
		\includegraphics[width=12cm]{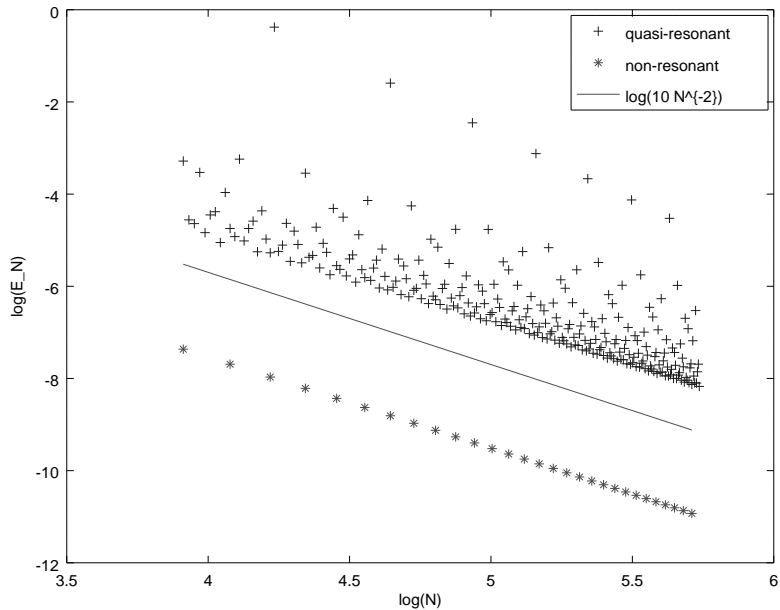}
	\end{center}
	\caption{\label{convergence_curves_ordre_2} Convergence curves with $u(x)=x(1-x)e^{2x}$ , $n=2$ and $E_N:=\|\mathbf{u}^N-\mathbf{u}^{N,ex}\|_{\infty}$. For the non-resonant scheme $d=2\mathbb{1}_{\{ 0 \} }-\mathbb{1}_{\{-1,1\}}$ and for the quasi-resonant scheme $d=(2-6z) \mathbb{1}_{\{ 0 \} }+(4z-1)\mathbb{1}_{\{-1,1\}}-z \mathbb{1}_{\{-2,2\}}$ with $z= 0.358946420670826$.}
\end {figure}

\begin {figure}[!ht]
	\begin{center}
		\includegraphics[width=12cm]{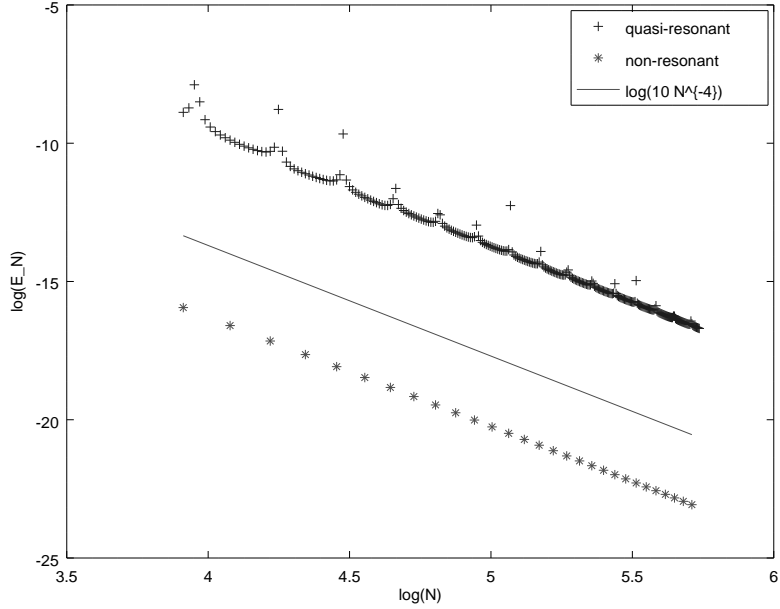}
	\end{center}
	\caption{\label{convergence_curves_ordre_4} Convergence curves with $u(x)=x(1-x)e^{2x}$ , $n=4$ and $E_N:=\|\mathbf{u}^N-\mathbf{u}^{N,ex}\|_{\infty}$. For the non-resonant scheme $d=2\mathbb{1}_{\{ 0 \} }-\mathbb{1}_{\{-1,1\}}$ and for the quasi-resonant scheme $d=(2-6z) \mathbb{1}_{\{ 0 \} }+(4z-1)\mathbb{1}_{\{-1,1\}}-z \mathbb{1}_{\{-2,2\}}$ with $z= 32.12121212$.}
\end {figure}

\section{Polynomials and high order formulas}
The aim of this section is two fold. First, to explain why the matrices $\mathbf{D}_N(d)$ constructed in Proposition \ref{le schema de base} are polynomials in $d$. Second, to give criteria of consistency on these polynomials to interpret Theorem \ref{theo les formules optimales} as a classical problem of Pad\'e approximant.

\medskip

To highlight the algebraic structure of the matrices $\mathbf{D}_N(d)$ of Proposition \ref{le schema de base} we will now give a more formal definition of these matrices. 

\medskip

Let $d \in \mathscr{S}_{\mathbb{C}}$ be a symmetric formula. 
We introduce $T_d\in \mathscr{L}(\mathbb{C}^\mathbb{Z})$, the operator of convolution by $d$,
\begin{equation}
\label{def convolution operator}
 \forall w\in \mathbb{C}^\mathbb{Z}, \ T_d(w)= d\star w = \left( \sum_{j\in \mathbb{Z}} d_j w_{i-j} \right)_{i\in \mathbb{Z}}.
\end{equation}
Let $N\in \mathbb{N}^*$  and $\mathscr{E}_N$ be the space of the odd functions from $\mathbb{Z}$ to $\mathbb{C}$ that are odd in $0$ and in $N+1$
\begin{equation}
\label{espace_fonctions_periodiques_impaires}
\mathscr{E}_N := \{ w\in \mathbb{C}^{\mathbb{Z}} \ | \ \forall j \in \mathbb{Z}, \ w_{N+1+j}=-w_{N+1-j} \textrm{ and } w_{-j}=-w_{j} \}.
\end{equation}
Let $\mathscr{B}_N$ be the canonical basis of $\mathscr{E}_N$
\begin{equation}
\label{la base canonique}
\mathscr{B}_N = (\mathbb{1}_{j+(2N+2)\mathbb{Z}}-\mathbb{1}_{-j+(2N+2)\mathbb{Z}})_{j=1,\dots, N}.
\end{equation}
Since $d$ is symmetric, we verify that $\mathscr{E}_N $ is stable by $T_d$.
\begin{defi} 
\label{def formal D_N}
 With the previous construction, we define $\mathbf{D}_N(d)$ through the relation
\begin{equation*}
\mathbf{D}_N(d) = \matr_{\mathscr{B}_N} T_{d | \ \mathscr{E}_N}. 
\end{equation*}
\end{defi}
\begin{rem}
Of course, this definition of $\mathbf{D}_N(d)$ gives the same matrices, when $N$ is large enough, as the matrix of Proposition \ref{le schema de base}.
\end{rem}

\medskip

The space of symmetric formulas has a structure of free module on the ring of polynomial that is very useful to write efficient and accurate high order schemes. More precisely, we equip the set of formula $\mathbb{C}^{(\mathbb{Z})}$ of its structure of commutative algebra for the convolution
\[ \forall d,s\in  \mathbb{C}^{(\mathbb{Z})}, \ d\star s = \left( \sum_{j\in \mathbb{Z}} d_j s_{i-j} \right)_{i\in \mathbb{Z}}.\]
Then, if $R$ is a ring such that $\mathbb{Z}\subset R\subset \mathbb{C}$, we consider $\mathscr{S}_{R}$ as a subalgebra of $\mathbb{C}^{(\mathbb{Z})}$.

\medskip

On the one hand, this structure explains, through the following lemma, the importance of the formula $a$ (defined in \eqref{definition_formule_du_prince} by $a= 2\mathbb{1}_{\{0\}} - \mathbb{1}_{\{-1,1\}}$).
\begin{lem} 
\label{lemme toute formule est un polynome du prince}
If $R$ is a ring such that $\mathbb{Z}\subset R\subset \mathbb{C}$ then $\mathscr{S}_{R}$ is a free $R[X]$ module whose $a$ is a basis
\[ \forall d\in  \mathscr{S}_{R},\ \exists \ ! \ P \in R[X], \ d=P(a).\]
Furthermore, if  $P \in \mathbb{C}[X]$ then 
\[ \tau(P(a)) = \deg P. \]
\end{lem}
\begin{proof}
If we consider $\mathbb{C}^{(\mathbb{Z})}$ as a subalgebra of $\mathbb{C}^{\left( \frac{\mathbb{Z}}2 \right)}$ then we remark that
\[ a = -\left( \mathbb{1}_{\left\{\frac12 \right\}} -\mathbb{1}_{\left\{-\frac12\right\}}  \right)^{\star 2}.\]
Consequently, a binomial expansion gives
\[ \forall n\in \mathbb{N}, \ a^{\star n} = (-1)^n \left( \mathbb{1}_{\left\{\frac12 \right\}} -\mathbb{1}_{\left\{-\frac12\right\}}  \right)^{\star 2n}  =\sum_{k=0}^{n} \frac{(2n)!}{(n+k)!(n-k)!} (-1)^{k} \mathbb{1}_{\{k,-k\}}.\]
The second point of the lemma is clearly a consequence of this expansion. Furthermore, since the term associated to the highest index of $a^{\star n}$ (i.e. $(-1)^n$) is invertible in $R$, the first point follows from an induction.
\end{proof}
On the other hand, this structure explains, why the matrices $\mathbf{D}_{N}(d)$ are polynomials in $\mathbf{A}_{N}$.
\begin{lem}
\label{lemme D_N est un morphisme de module}
For all $N\in \mathbb{N}^*$, $d\mapsto \mathbf{D}_N(d)$ is a $\mathbb{C}[X]$ module morphism:
\[ \forall P\in \mathbb{C}[X], \forall d\in \mathscr{S}_{\mathbb{C}}, \  \mathbf{D}_N(P(d))= P(\mathbf{D}_N(d)). \]
\end{lem}
\begin{proof}
It follows directly of  Definition \ref{def formal D_N} of $\mathbf{D}_N(d)$ and of the associativity of the convolution.
\end{proof}



\subsection{Consistency for the polynomials}
We start with a lemma that we have used implicitly in the introduction (in Proposition \eqref{prop si on a d alors on a s} and Proposition \eqref{prop de D_N}).
\begin{lem}
\label{lem pour avoir la consistance, il suffit de s'exiter sur les polynomes}
Let $n\in 2\mathbb{N}$. Then a couple of formulas $(d,s)\in \mathscr{S}_{\mathbb{C}}^2 $ is consistent of order $n$ \eqref{def consistency formula} if and only if
\[ \forall p\in \mathbb{C}_{n+1}[X], \ \sum_{j\in \mathbb{Z}} d_j p(j) + s_j p''(j) =0. \]
\end{lem}
\begin{proof}
It is enough to choose $u(x)=x^i$ with $i\leq n+1$ is the definition of the consistency and then to simplify the powers of "$h$". Conversely, it is enough to do a Taylor expansion.
\end{proof}

\medskip

In particular, if we choose $p=1$, we find that the consistency of order $n=0$ is nothing but the condition of zero mean \eqref{zero mean} for $d$.

\medskip

We introduce the formal Fourier transform $\mathscr{F}$ from the algebra of formulas $\mathbb{C}^{(\mathbb{Z})}$ to the algebra of formal series $\mathbb{C}\llbracket X \rrbracket$ defined by
\[ \mathscr{F}: \left\{ \begin{array}{llll} \mathbb{C}^{(\mathbb{Z})} & \to & \mathbb{C}\llbracket X \rrbracket\\
															d & \mapsto & \displaystyle \sum_{j\in \mathbb{Z}} d_j e^{ijX}.
\end{array}  \right. \]

\medskip

We give a characterization of consistency through this transform.
\begin{lem}
\label{lem consistency en mode Fourier}
Let $n\in 2\mathbb{N}$. A couple of formulas $(d,s)\in \mathscr{S}_{\mathbb{C}}^2 $ is consistent of order $n$ \eqref{def consistency formula} if and only if
\[  \mathscr{F}d = X^2 \mathscr{F}s \mod X^{n+2}.  \]
\end{lem}
\begin{proof}
Let $\partial_X\in \mathscr{L}(\mathbb{C}[X])$ be the formal derivative on the space of the polynomials $\mathbb{C}[X]$. As a consequence, since the Taylor expansion in $0$ of a polynomial is exact, if $p\in \mathbb{C}[X]$ and $x_0\in \mathbb{C}$ then we have
\[ p(x_0) = e^{x_0 \partial_X}p(0).  \]
Consequently, following Lemma \eqref{lem pour avoir la consistance, il suffit de s'exiter sur les polynomes}, $(d,s)\in \mathscr{S}_{\mathbb{C}}^2 $ is consistent of order $n$ \eqref{def consistency formula} if and only if
\[ \forall p\in \mathbb{C}_{n+1}[X], \ \left( \mathscr{F}d - X^2 \mathscr{F}s \right)(i\partial_X)p(0)=0. \]
We conclude the proof by considering the lowest power in the expansion of $\mathscr{F}d - X^2 \mathscr{F}s$ in $0$.
\end{proof}

\medskip

The formal Fourier transform is as usual an algebra morphism. As a consequence, if $P\in \mathbb{C}[X]$ and $d=P(a)$ then
\begin{equation}
\label{lien fourier polynomes}
\sum_{k \in \mathbb{N}^*} \sum_{j\in \mathbb{Z}} d_j \frac{(-1)^k j^{2k}}{(2k)!} X^{2k} =  \mathscr{F}d = P(\mathscr{F}a) = P(2-2\cos(X)) = P\left( 4\sin^2\left( \frac{X}2 \right) \right).
\end{equation}
The consistency and the stability of the formulas often involve moment of $d$ or $s$. In particular, this relation provides simple expressions for the first moments of $d$ in function of $P$
\[  \sum_{j\in \mathbb{Z}} d_j = P(0) \textrm{ and }  \sum_{j\in \mathbb{Z}} d_j j^2 = - 2 P'(0).\]

\medskip

In fact, with the formula \eqref{lien fourier polynomes}, we get a criterion of consistency directly on the polynomials.
\begin{lem}
\label{lem consistency arcsin}
Let $P,Q\in \mathbb{C}[X]$ and $n\in 2\mathbb{N}^*$. The couple of symmetric formulas $(P(a),Q(a))$ is consistent of order $n$ \eqref{def consistency formula} if and only if 
\[  P(4 X^2) =4 \left( \arcsin(X) \right)^2 Q(4 X^2)  \mod X^{n+2},\]
where $\left( \arcsin(X) \right)^2$ is the square of the inverse sine function whose expansion is (for a reference, see, for example, \cite{MR2320777})
\[ \left( \arcsin(X) \right)^2 = \sum_{n\in \mathbb{N}^*} \frac{2^{2n-1} }{n^2 C_{2n}^n} X^{2n}.   \]
\end{lem}
\begin{proof}
If we apply \eqref{lien fourier polynomes} to the criterion of consistency of Lemma \ref{lem consistency en mode Fourier} then it comes
\[ P\left( 4\sin^2 \left( \frac{X}2 \right) \right) = X^2 Q\left( 4\sin^2 \left( \frac{X}2 \right) \right) \mod X^{n+2} . \]
To conclude the proof, it is enough to do the change of variable
\[ X\leftarrow 2 \arcsin(X). \]
\end{proof}

\subsection{The optimal case}
In order to prove Theorem \ref{theo les formules optimales} we are going to explain the link between the problem of optimization \eqref{le pb dopti} and the theory of Padé approximant. To see this link we introduce the usual valuation on $\mathbb{C}\llbracket X \rrbracket$:
\[ \forall C \in \mathbb{C}\llbracket X \rrbracket, \ {\rm val}(C) = \min \{k \in \mathbb{N} \ | \ \forall 0\leq j \leq k, \ C^{(j)}(0)=0 \}.  \]
As a consequence, with this formalism, Lemma \ref{lem consistency arcsin} can be written
\[\forall P,Q\in \mathbb{C}[X],\  \ord \left( P(a),Q(a) \right) = {\rm val} \left( P(4 X^2) - 4\left( \arcsin(X) \right)^2 Q(4 X^2) \right) -2. \]
However, Lemma \ref{lemme toute formule est un polynome du prince} proves that $(P,Q)\mapsto (P(a),Q(a))$ is a bijection from $\mathbb{C}_{l+1}[X]\times \mathbb{C}_{m}[X]$ to the space of the couples of symmetric formulas $(d,s)$ such that $\tau(d)\leq l+1$ and $\tau(s)\leq m$. As a consequence, the problem of optimization \eqref{le pb dopti} is equivalent to the following
 \[ \mathop{\max_{ (P,Q) \in \mathbb{C}[X]^2 \setminus \{(0,0)\}  }}_{\deg P \leq l+1,\ \deg Q \leq m} {\rm val} \left( P(4 X^2) - 4\left( \arcsin(X) \right)^2 Q(4 X^2) \right).\]
Since if $P(0)\neq 0$ then ${\rm val} \left( P(4 X^2) - 4\left( \arcsin(X) \right)^2 Q(4 X^2) \right)=0$, it is natural to study this problem of optimization for polynomials $P$ such that 
\[ P=XR \textrm{ where } R \in \mathbb{C}_{l}[X].\]
Consequently, it is enough to study the following problem of optimization
\begin{equation}
\label{le pb dopti Pade}
  \mathop{\max_{ (R,Q) \in \mathbb{C}[X]^2 \setminus \{(0,0)\}  }}_{\deg R \leq l,\ \deg Q \leq m} {\rm val} \left( R - C Q \right), 
\end{equation}
with (see \cite{MR2320777} for the expansion)
\begin{equation}
\label{arcsin_rescalle}
C(X) := 4 \left( \frac{\arcsin( \frac{\sqrt{X}}2)}{\sqrt X} \right)^2 = 2 \sum_{n\in \mathbb{N}} \frac{ X^{n} }{(n+1)^2 C_{2n+2}^{n+1}} . 
\end{equation}

\medskip
The theory of Padé approximants is a deep theory about approximation of formal series by rational ones. It has been extensively developed in the last decades (see  \cite{MR1383091}  or  \cite{MR513420} for an overview). Its aim is to give to each formal series $F\in \mathbb{C}\llbracket X \rrbracket$ a rational approximation $\frac{p_{l,m}}{q_{l,m}}$ (usually noted $[l/m]$) such that
\begin{equation}
\label{def Pade intuitif}
F = \frac{p_{l,m}}{q_{l,m}} \mod X^{l+m+1} \textrm{ , with } p_{l,m} \in \mathbb{C}_{l}[X] \textrm{ and }q_{l,m} \in \mathbb{C}_{m}[X].
\end{equation}
A natural way to find such an approximation is to try to solve
\begin{equation}
\label{def Pade usuel}
 p_{l,m} = Fq_{l,m}  \mod X^{l+m+1}  \textrm{ , with } p_{l,m} \in \mathbb{C}_{l}[X] \textrm{ and }q_{l,m} \in \mathbb{C}_{m}[X].
\end{equation}
Indeed, if we get a solution $(p_{l,m},q_{l,m})$ of \eqref{def Pade usuel} with $q_{l,m}(0)\neq 0$ then it is also a solution of \eqref{def Pade intuitif}. The second formulation \eqref{def Pade usuel} is interesting because it is a linear system of $l+m+1$ equations and $l+m+2$ unknowns. Consequently, it admits at least one non trivial solution. However, the question of its uniqueness (up to multiplication by a scalar) is generaly non trivial. In the classical Padé theory, if for a formal series $F$, the linear system \eqref{def Pade usuel} admits for all $l,m\in \mathbb{N}$, a unique non trivial solution (up to multiplication by a scalar), then it is said that the Padé table of $F$ is {\em normal}. Furthermore, if $F(0)\neq 0$ and if its Padé table is normal then a non trivial solution $(p_{l,m},q_{l,m})$ of \eqref{def Pade usuel} satisfies $\deg p_{l,m}=l$, $\deg q_{l,m}=m$, $q_{l,m}(0)\neq 0$ and ${\rm val} ( p_{l,m} - Fq_{l,m})=l+m+1$ (see  \cite{MR1383091}  or  \cite{MR513420} for details).

\medskip

What is crucial for us is that the Pad\'e table of $C$ is normal. In fact, D. Karp and E. Prilepkina have proved in \cite{KARP2012348} that the Padé tables of many generalized hypergeometric functions are normals. To see that the Padé table of $C$ is normal, we just have to verify that $C$ is one of those generalized hypergeometric functions. The generalized hypergeometric functions are the formal series defined by
\begin{equation}
\label{hendrix}
\pFq{p}{q}{\alpha_1, \dots, \alpha_p}{\beta_1, \dots , \beta_q}{X} = \sum_{k\in \mathbb{N}} \frac{(\alpha_1)_k\dots (\alpha_p)_k}{(\beta)_1 \dots (\beta_q)_k} \frac{X^k}{k!} \textrm{ with } (\gamma)_k = \prod_{j=0}^{k-1} \gamma +j.   
\end{equation}
D. Karp and E. Prilepkina have proved in Theorem $9$ of \cite{KARP2012348} that if 
\[ \left\{ \begin{array}{lll} p=q+1, \\
								  0<\alpha_{q+1}\leq 1,\\
								   0<\alpha_1 \leq \dots \leq \alpha_q,\\
								   0<\beta_1 \leq \dots \leq \beta_q,\\
								   \forall k\in \llbracket 1,q \rrbracket, \ \displaystyle \sum_{j=1}^k \alpha_j \leq \sum_{j=1}^k \beta_j
\end{array}   \right. \]
then the Pad\'e table of $\pFq{p}{q}{\alpha_1, \dots, \alpha_p}{\beta_1, \dots , \beta_q}{X}$ is normal. However, $C$ is one of those generalized hypergeometric functions because
\begin{equation}
\label{C est une fonction hypergeometrique}
C(X) =  \pFq{3}{2}{1,1,1}{\frac32,2}{\frac{X}4}.
\end{equation}
We verify this assertion by the following elementary calculation
\[ \frac{(n+1)^2 C_{2n+2}^{n+1}}{(n+2)^2 C_{2n+4}^{n+2}} = \frac{ (n+1)^2  }{ (2n+3)(2n+4) } = \frac14 \frac{ (n+1)^2 }{(n+\frac32)(n+2)}, \]
which shows by induction that the coefficients of $C(X)$ (see \eqref{arcsin_rescalle}) coincide with those of one of the generalized hypergeometric functions in \eqref{C est une fonction hypergeometrique}, see \eqref{hendrix}. 
\medskip

Now, we just have to link these results of Pad\'e approximation with our optimization problem \eqref{le pb dopti}. But if we denote by $(R_{l,m},Q_{l,m})$ the solution of \eqref{def Pade usuel} such that $R_{l,m}(0)=1$, then we have 
\[ {\rm val}(R_{l,m} - C RQ_{l,m})=l+m+1.\]
Conversely, if $(R,Q)$ satisfyies ${\rm val}(R - C RQ)\geq l+m+1$ with $\deg R\leq l$ and $\deg Q\leq m$ then it is a solution of \eqref{def Pade usuel}. But since the Pad\'e table of $C$ is normal, $(R,Q)$ is equal to $(R_{l,m},Q_{l,m})$, up to multiplication by a scalar. 

\medskip

Consequently, we have proved that the numerator and the denominator of the Pad\'e approximant of $C$ are the solutions to the optimization problem \eqref{le pb dopti Pade}, up to multiplication by a scalar. All the results of this analysis is summarized in the following theorem that is nothing but a  version of Theorem \ref{theo les formules optimales} with polynomials.
\begin{theo}
 \label{theo les poly optimaux}
For all $l,m\in \mathbb{N}$, there exists a couple of rational polynomial $(R_{l,m},Q_{l,m})\in \mathbb{Q}[X]^2$ such that
\[ \left\{  \begin{array}{llll} \deg R_{l,m}=l, \\
									\deg Q_{l,m}=m,\\
									R_{l,m}(0)=1.
\end{array} \right. \]
Moreover $(R_{l,m},Q_{l,m})$ is solution of the optimization problem
\[ \mathop{\max_{ (R,Q) \in \mathbb{C}[X]^2 \setminus \{(0,0)\}  }}_{\deg R \leq l,\ \deg Q \leq m} {\rm val} \left( R - C Q \right) =  {\rm val} \left( R_{l,m} - C Q_{l,m} \right)= l+m+1.\]
Furthermore, this solution is essentially unique: if $(R,S)\in \mathbb{C}[X]^2 $ is such that $\deg R\leq l$, $\deg Q\leq m$ and $ {\rm val} \left( R_{l,m} - C Q_{l,m} \right)= l+m+1$ then there exists $\lambda\in \mathbb{C}$ such that $R=\lambda R_{l,m}$ and $Q=\lambda Q_{l,m}$.
\end{theo}

\medskip

There exists many very efficient methods to compute effectively Padé approximants (see for example \cite{MR1383091} or \cite{MR513420}). Consequently, if the order of consistency is large enough, it is interesting to not compute the optimal formulas of Theorem \ref{theo les formules optimales} through the resolution of the linear system \eqref{linear system for optimal formulas}, but to compute them from the optimal polynomials of Theorem \ref{theo les poly optimaux} through the relations
\begin{equation}
\label{lien formules polynomes optimaux}
 s^{l,m}=Q_{l,m}(a) \textrm{ and } d^{l,m}=P_{l,m}(a) \textrm{ with } P_{l,m}(X)=XR_{l,m}(X).
\end{equation}

\section{Stability}
In this section we study criteria of stability for the sequences of matrices of the form $P(\mathbf{A}_N)$ with $P$ a polynomial. These conditions hold on the polynomial $P$. As a consequence, if we want to apply one of these criteria to a matrix of the form $\mathbf{D}_N(d)$, with $d$ a symmetric formula, we have to solve  $P(a)=d$ (see  Lemma \ref{lemme toute formule est un polynome du prince} for details).

\medskip

In the first part, we give a criterion of strong stability \eqref{def strong stability} and then we deduce Theorem \ref{theo les schemas optimaux sont stables} and the first part of Theorem \ref{theo alea} (when the formulas are complex). In the second part, we give a diophantine criterion of relative stability \eqref{def relative stability} that is enough to prove the second part of Theorem \ref{theo alea} (when the formulas are real).

\subsection{Strong stability}
\begin{theo}
\label{theorem_positive_case} 
Let $P\in \mathbb{C}[X]$ be a polynomial such that 
\begin{equation}
\label{assumption_theorem_positive}
P(0)=0, \ P'(0)\neq 0  \textrm{ and }  \forall x\in ]0,4], \ P(x)\neq 0.
\end{equation}
Then the sequence of matrices $(P(\mathbf{A}_N))_{N \in \mathbb{N}^*}$ is strongly stable \eqref{def strong stability}.
\end{theo}
\begin{proof}
The assumptions \eqref{assumption_theorem_positive} implies that there exists $\beta\neq 0$ a real number and a sequence $(\mu_k)_{k=1\dots d}$ of complex numbers such that
\[ P(X) =\beta X \prod_{k=1}^d \left( X-\mu_k \right).\]
On the one hand, a straightforward calculation shows that $\mathbf{A}_N$ is invertible and 
\[\forall i,j\in \llbracket 1,N \rrbracket, \ (\mathbf{A}_N^{-1})_{i,j} = \min(j(1-hi),i(1-hj)). \]
On the other hand, since by assumption $\mu_k\notin [0,4]$, the following lemma (proved in Appendix \ref{proof_lemma_op_borne}) shows that $\mathbf{A}_N-\mu_k\mathbf{I}_N$ is invertible and that there exists a constant $c_{\mu_k}$ such that for all $N$, $\|(\mathbf{A}_N-\mu_k\mathbf{I}_N)^{-1}\|_{\infty}\leq c_{\mu_k}$.
\begin{lem}
\label{lemma_op_borne}
If $\mu\in \mathbb{C}\setminus [0,4]$ then there exists $c>0$ such that for all $N\in \mathbb{N}^*$, $\mathbf{A}_N-\mu \mathbf{I}_N$ is invertible and for all $\mathbf{v}\in \mathbb{C}^N$
\[ \| \mathbf{v} \|_{\infty}  \leq c \|\mathbf{A}_N \mathbf{v}-\mu \mathbf{v}\|_{\infty}. \]
\end{lem}
\noindent As a consequence, $P(\mathbf{A}_N)$ is invertible and we have
\[ \forall N\in \mathbb{N}^*, \forall \mathbf{v} \in \mathbb{R}^N, \ \|P(\mathbf{A}_N)^{-1} \mathbf{v}\|_{\infty} \leq |\beta|^{-1} \| \mathbf{A}_N^{-1} \mathbf{v}\|_{\infty} \prod_{k=1}^d c_{\mu_k}.\]
Hence, to prove Theorem \ref{theorem_positive_case}, it is enough to prove that $(\mathbf{A}_N)_N$ is strongly stable \eqref{def strong stability}. The estimation of strong stability of $(\mathbf{A}_N)_N$ is very explicit and is given for $l\in \mathbb{N}$ by
\begin{align*}
&\|\mathbf{A}_N^{-1} \mathbf{v} \|_{\infty}\\
 &\leq \sup_{i=1}^N \sum_{j=1}^N |\mathbf{v}_j|  \min\{ i(1-hj),j(1-hi) \}\\
& \leq  \sup_{i=1}^N \sum_{j \in \llbracket l+1,N-l\rrbracket } |\mathbf{v}_j|  \min\{ i(1-hj),j(1-hi) \} + \sup_{i=1}^N \sum_{j \in \llbracket l+1,N-l\rrbracket ^c }  |\mathbf{v}_j|  \min\{ i(1-hj),j(1-hi) \}  \\
&\leq  \sup_{i=1}^N \sum_{j \in \llbracket l+1,N-l\rrbracket } |\mathbf{v}_j|  \frac4h+ \sup_{i=1}^N \sum_{j \in \llbracket l+1,N-l \rrbracket ^c }  |\mathbf{v}_j|  l  \\
&\leq \sup_{j=1}^N \left\{  \begin{array}{lll} 4 h^{-2} |\mathbf{v}_j| & \textrm{ if }  l+1 \leq j\leq N-l,\\
															2l^2  |\mathbf{v}_j| & \textrm{ else.}
\end{array} \right.
\end{align*}
\end{proof}

\medskip

\noindent \bf Application 1: Proof of the first part of Theorem \ref{theo alea}.\rm \\
The more direct application of this criterion of strong stability is the first part of Theorem \ref{theo alea}. Since we have proved in Lemma \ref{lemme toute formule est un polynome du prince} that $P\mapsto P(a)$ induce an isomorphism of vector space between $X\mathbb{C}_{l}[X]$ and $ \mathscr{C}_{\mathbb{C},l}$, it is enough to prove that almost all complex polynomials of degree smaller than $l+1$ do not have any zero point in $[0,4]$ to conclude with the criterion of stability of Theorem \ref{theorem_positive_case}. In fact, we show that almost all complex polynomials of degree smaller than $l+1$ do not have any real zero point.

\begin{proof} Since the null sets are the same for all the Lebesgue measures on $\mathbb{C}_{l}[X]$ it is enough to prove the result for one well chosen Lebesgue measure.
 As a consequence, we introduce $\lambda$ be a Lebesgue measure on $\mathbb{R}_{l}[X]$ and we consider $\lambda^{\otimes 2}$ as a Lebesgue measure on $\mathbb{C}_l[X]$ induced by the direct sum 
\[ \mathbb{C}_l[X] = \mathbb{R}_l[X] \oplus i \mathbb{R}_l[X].\]
Now, we remark that if a polynomial $P\in \mathbb{C}_l[X]$ admits the decomposition $P=P_1+iP_2$ and has a real zero point $x \in \mathbb{R}$ then $x$ is a zero point of $P_1$ and of $P_2$. As a consequence, we conclude by the following calculation
\begin{align*}
 \lambda^{\otimes 2} \{ P\in  \mathbb{C}_l[X] \ | \ \exists x\in \mathbb{R}, \ P(x)=0  \} &= \int_{\mathbb{R}_l[X]}  \int_{\mathbb{R}_l[X]} \mathbb{1}_{\exists x\in \mathbb{R}, \ (P_1+iP_2)(x)=0} { \rm d \lambda(P_1)  } { \rm d \lambda(P_2)  }\\
&=  \int_{\mathbb{R}_l[X]}  \int_{\mathbb{R}_l[X]}  \mathbb{1}_{\exists x\in \mathbb{R}, \ P_2(x)=P_1(x)=0}  { \rm d \lambda(P_1)  } { \rm d \lambda(P_2)  }\\
&\leq  \int_{\mathbb{R}_l[X]} \sum_{x\in \mathbb{R}, P_2(x)=0} \int_{\mathbb{R}_l[X]}  \mathbb{1}_{ P_1(x)=0} { \rm d \lambda(P_1)  } { \rm d \lambda(P_2)  }\\
&=0.
\end{align*}
The last equality is nothing but, since $\{ P_1 \in \mathbb{R}_l[X] | \ P_1(x)=0 \}$ is an hyperplane of $\mathbb{R}_l[X]$, its Lebesgue measure is zero.
\end{proof}

\medskip

\noindent \bf Application 2: Proof of Theorem \ref{theo les schemas optimaux sont stables} .\rm \\
The second application of the criterion of strong stability of Theorem \ref{theorem_positive_case} is the Theorem \ref{theo les schemas optimaux sont stables} about the strong stability of the most efficient schemes. In fact, to apply this criterion to the optimal formulas of Theorem \ref{theo les formules optimales}, we exactly have to prove that the optimal polynomials $R_{l,m}$ of Theorem \ref{theo les poly optimaux} do not have any zeros point in $[0,4]$.
\begin{proof}
Let $l,m\in \mathbb{N}$ be some integers and $R_{l,m}$ the optimal polynomial given by Theorem \ref{theo les poly optimaux}. In the proof of this theorem, $R_{l,m}$ is built as the numerator of the Padé approximant of the function $C$ \eqref{arcsin_rescalle}. Futhermore, as we have explained in the proof of Theorem \ref{theo les poly optimaux},  D. Karp and E. Prilepkina have proved in \cite{KARP2012348} that $C(-X)$ is a Stieltjes transform of a measure supported in $[0,4]$. As a consequence, we can use the classical results about the localization of the zeros points and poles of the Padé approximants of such series.

\medskip

On the one hand, it is enough to apply the point $(vii)$ of Theorem $3$ page $251$ of the book of J. Gilewicz \cite{MR513420} to prove that if $k\leq  0$ and $l\geq -k$ then all the zero points of $R_{l+k,l}$ are in $]4,\infty[$.

\medskip
On the other hand, J. Gilewicz proves at the point $(iii)$ of this theorem that if $k\geq  -1$, $l+k\geq 0$ and $l\geq 0$ then all the zero points of $Q_{l+k,l}$ (the denominator of the Padé approximant of $C$) are in $]4,\infty[$. Futhermore, page $264$ of his book \cite{MR513420}, J. Gilewicz gives a theorem of Stieltjes and Wynn (point $(iii)$ of Theorem $5$) that implies that if $k \geq 0$ and $l\leq 0$ then
\[  \forall x\in [0,4], \ \frac{R_{l+k,l}(x)}{Q_{l+k,l}(x)}\leq \frac{R_{l+k+1,l+1}(x)}{Q_{l+k+1,l+1}(x)}.\]
Since $Q_{l+k,l}$ does not have any zero point on $[0,4]$ and since by construction $Q_{l+k,l}(0)=R_{l+k,l}(0)=1$ then it follows that for all $k\geq 0$ and all $l\geq 0$ we have
\[ \forall x\in [0,4], \ Q_{l+k,l}(x)>0.\]
As a consequence, if $k \geq 0$ and $l\geq 0$ then, we have
\[  \forall x\in [0,4], \ \frac{Q_{l+k+1,l+1}(x)}{Q_{l+k,l}(x)} R_{l+k,l}(x) \leq R_{l+k+1,l+1}(x).\]
 Consequently, if for all $k\geq 0$, we prove that $R_{k,0}$ is positive on $[0,4]$, then we conclude by induction on $l\geq 0$, that $R_{l+k,l}$ is positive on $[0,4]$. Indeed, it is clear that $R_{k,0}$ is positive on $[0,4]$ because by construction (see Theorem \ref{theo les poly optimaux}), we have
 \[ R_{k,0} =  2 \sum_{n=0}^k \frac{ X^{n} }{(n+1)^2 C_{2n+2}^{n+1}} >0 \textrm{ on } \mathbb{R}^+. \]
\end{proof}

\subsection{Relative stability}

\begin{theo}
\label{le_theoreme_diophantien}
Let $P\in \mathbb{C}[X]$ be a polynomial and let $\Lambda$ be the set of the roots of $P$ in $[0,4]$ and assume that $P$ satisfies the following assumptions:
\begin{enumerate}[i)]
\item $0\in \Lambda$,
\item $4\notin \Lambda$,
\item the roots of $P$ in $[0,4]$ are simple,
\item $\exists \delta: \mathbb{N}^* \to \mathbb{R}_+^*$,
\begin{equation}
\label{diophantine_assumptions}
 \forall \lambda\in \Lambda, \forall q\in \mathbb{N}^*, \forall 1\leq p\leq  q-1, \  \ 0<\delta_q\leq \left| \lambda- 4\sin^2 \left( \frac{\pi}2 \frac{p}q  \right)  \right|.
\end{equation}
\end{enumerate}
Then the sequence of finite difference matrices $(P(\mathbf{A}_N))_{N \in \mathbb{N}^*}$ is stable relatively to the sequence $\eta_N = \frac1{\delta_{N+1}}$ \eqref{def relative stability}.
\end{theo}
\begin{proof}
see Appendix \ref{proof of the Diophantine criterion}.
\end{proof}

\medskip

\noindent \bf Application 1: stability for second order algebraic zero points\rm \\
The first application of this diophantine criteria is based on a classical result about approximation of algebraic numbers by rational ones. It gives a way to design sequences of matrices $\mathbf{D}_N$ that are stable \eqref{def stability},  but such that $\mathbf{D}_N$ has not a positive or a negative spectrum for all $N$.
\begin{theo} Liouville's Theorem. (from the book of Andrei B. Shidlovskii \cite{MR1033015} page 23)\\
If $\alpha$ is a real algebraic number of degree $n$, $n\geq 1$, then there exists a constant $c=c(\alpha)>0$ such that the following inequality holds for any $p\in \mathbb{Z}$ and
$q\in \mathbb{N}^*$, $\frac{p}{q}\neq \alpha$:
\[ \left| \alpha - \frac{p}q \right| > \frac{c}{q^n}. \]
\end{theo}
\begin{corol} If a polynomial $P\in \mathbb{C}[X]$ satisfies the three first hypothesis of Theorem \ref{le_theoreme_diophantien} and if for all root $\lambda\in \Lambda\setminus\{0\}$ there exist an algebraic number of degree $2$, $\alpha$, such that $\lambda= 4\sin^2(\frac{\pi}2 \alpha)$, then the sequence of finite difference matrices $(P(\mathbf{A}_N))_N$ is stable \eqref{def stability}.
\end{corol}

\medskip

\noindent \bf Application 2: Proof of the second part of Theorem \ref{theo alea}.\rm \\
The proof of the second part of Theorem \ref{theo alea} is an adaptation of a classical qualitative result about approximation of real numbers by rational ones. 
\begin{theo} A version of the Khinchin's Theorem. (see for example \cite{MR1033015} page 17)\\
\label{Khinchin's Theorem} 
	Let $(\nu_{q})_q$ be a sequence of positive real numbers such that the series $\sum \nu_{q}$ converges. Then, for almost all $\alpha\in \mathbb{R}$, there exists a constant $c>0$ such that for all $p,q\in \mathbb{Z}\times \mathbb{N}^*$, one has
	\[ |\alpha -\frac{p}q| \geq c \frac{\nu_{q}}{q}.\]
\end{theo}
More precisely, to prove the second part of Theorem \ref{theo alea} with the criterion of Theorem \ref{le_theoreme_diophantien}, it is enough to prove that the following set are null set for a Lebesgue measure on $\mathbb{R}_l[X]$ (they are the sets of the polynomials that do not satisfy $ii, iii$ or $iv$):
\[ E_1 = \{ R \in \mathbb{R}_l[X] \ | \ R(4)=0 \textrm{ or } R(0)=0  \}, \]
\[ E_2 =  \{R \in  \mathbb{R}_l[X] \ | \  \exists \lambda \in [0,4], \ R( \lambda)=R'( \lambda)=0  \}, \]
\[ E_3 = \left\{ R \in \mathbb{R}_l[X] \ | \  \exists \lambda \in [0,4], \ R(\lambda)=0 \textrm{ and } \liminf_{q \to \infty} \min_{p\in \llbracket 1,q-1\rrbracket} \eta_{q-1} \left| \lambda- 4\sin^2 \left( \frac{\pi}2 \frac{p}q  \right)  \right| = 0  \right\} . \]
Indeed, since we have proved in Lemma \ref{lemme toute formule est un polynome du prince} that $R\mapsto (XR)(a)$ induce an isomorphism of vector space between $\mathbb{R}_{l}[X]$ and $ \mathscr{C}_{\mathbb{R},l}$, the null sets for the Lebesgue measures on $\mathbb{R}_{l}[X]$ are associated to the null sets for the Lebesgue measures on $ \mathscr{C}_{\mathbb{R},l}$.

\medskip

It is quite clear that $E_1$ and $E_2 $ are null sets. Indeed, $E_1$ is a null set because since $P \mapsto P(4)$ is linear, it is an hyperplane and $E_2$ is a null set because it is the set of the zero points of the discriminant $\Delta(R)=\Res(R,R')$ that is a non zero polynomial of $R$. However, to prove that $E_3$ is a null set, we have to adapt the proof of the Khinchin's Theorem \ref{Khinchin's Theorem}.

\medskip

 In order to use the Borel Cantelli Theorem, we introduce a probability measure $\rho$ on $\mathbb{R}_l[X]$ with the same null set as a Lebesgue measure. 
 More precisely, we introduce the Lebesgue measure $\mu$ on $\mathbb{R}_l[X]$ induced by the Hardy's scalar product $\langle .,. \rangle_{\mathcal{H}^2}$. This scalar product is defined by
 \[ \forall R_1,R_2 \in \mathbb{R}_l[X], \ \langle R_1,R_2 \rangle_{\mathcal{H}^2} := \sum_{k=0}^l \frac{R_1^{(k)}(0) R_2^{(k)}(0)}{k!^2}.  \]
 Then, we define $\rho$ through its density with respect to $\mu$
 \[ \frac{ {\rm d \rho}  }{ {\rm d \mu}  } = \frac1{\sqrt{2\pi}^{l+1}} e^{ - \frac12 \| R  \|_{\mathcal{H}^2}^2} . \]
 As $\rho$ has a positive density with respect to $\mu$, $\rho$ and $\mu$ have the same null sets.\\
Hence, since $E_2$ is a null set, it is enough to prove that $E_3 \cap E_2^c$ is a null set. As a consequence, we can use the following inclusion
\[  E_2^c  \cap  E_3 \subset  E_2^c  \cap \left\{ R \in \mathbb{R}_l[X] \ | \ \liminf_{q \to \infty} \min_{p\in \llbracket 1,q-1\rrbracket} \eta_{q-1} \left| R \left( 4\sin^2 \left( \frac{\pi}2 \frac{p}q  \right) \right) \right| = 0  \right\}.\]
Then, we introduce the measurable sets
\[ F_q := \left\{ R \in \mathbb{R}_l[X] \ | \  \min_{p\in \llbracket 1,q-1\rrbracket}  \left| R \left( 4\sin^2 \left( \frac{\pi}2 \frac{p}q  \right) \right) \right| \leq \frac1{\eta_{q-1}}  \right\}, \]
to get the inclusion
\[ E_2^c  \cap  E_3 \subset E_2^c \cap  \limsup_{q \to \infty} F_q.\]
Consequently, it is enough to prove that $\sum \rho(F_q) <\infty$ to conclude by the Theorem of Borel Cantelli that $E_3$ is a null set.

\medskip

To control $\rho(F_q)$, we begin assuming the following lemma, that we will show at the end of this proof.
\begin{lem}
\label{lemma measure set}
 For all $\lambda\in \mathbb{R}$ and for all $\beta>0$, we have 
\[  \rho \left( \left\{  R \in \mathbb{R}_l[X] \ | \ |R(\lambda)| \leq \beta \right\} \right)\leq \sqrt{\frac{2}\pi} \beta . \]
\end{lem}

\medskip

Consequently, we deduce from the last assumption of Theorem \ref{theo alea} that $\sum \rho(F_q) <\infty$,
\begin{align*}
\rho(F_q) &\leq \sum_{p=1}^{q-1} \rho \left\{ R \in \mathbb{R}_l[X] \ | \   \left| R \left( 4\sin^2 \left( \frac{\pi}2 \frac{p}q  \right) \right) \right| \leq \frac1{\eta_{q-1}}  \right\} \\
			  &\leq (q-1) \sqrt{\frac{2}\pi} \frac1{\eta_{q-1}} \in l^1(\mathbb{N}\setminus \{0,1\})
\end{align*}

\medskip

To conclude this proof, we have to prove Lemma \ref{lemma measure set}. We introduce the polynomial $R_{\alpha}\in \mathbb{R}_l[X]$ defined by
\[ R_\alpha(X) = \sum_{k=0}^l (\alpha X)^k. \]
$R_{\alpha}$ is the Riesz representer of the evaluation in $\alpha$
\[ \forall R \in \mathbb{R}_l[X], \  R(\alpha) = \langle R_{\alpha} , R \rangle_{\mathcal{H}^2}. \]
Consequently, since the Gaussian measure $\rho$ is isotropic, we have
\begin{align*}
\rho \left( \left\{  R \in \mathbb{R}_l[X] \ | \ |R(\lambda)| \leq \beta \right\} \right) &= \rho \left( \left\{  R \in \mathbb{R}_l[X] \ | \ | \langle R_{\alpha} , R \rangle_{\mathcal{H}^2}| \leq \beta \right\} \right) \\
																											    &= \frac1{\sqrt{2\pi}} \int_\mathbb{R}  \mathbb{1}_{ |y| \| R_{\alpha}  \|_{\mathcal{H}^2}^2 \leq \beta }  e^{- \frac{y^2}2} {\rm dy} \\
																											    &\leq \frac1{\sqrt{2\pi}} \int_\mathbb{R} \mathbb{1}_{|y|  \| R_{\alpha}  \|_{\mathcal{H}^2}^2 \leq \beta} {\rm dy} \\
																											    &= \sqrt{\frac{2}\pi} \beta \left( \sum_{k=0}^l \alpha^{2k} \right)^{-1} \leq  \sqrt{\frac{2}\pi} \beta.
\end{align*}

\section{Appendix}
\subsection{Proof of Proposition \ref{prop the schemes are consistent}}
\label{proof the schemes are consistent}
Let $f\in C^{\infty}(\mathbb{R})$ be the source function of the Dirichlet problem \eqref{Dirichet_homogene}, and $u$ its solution. For each $N\in \mathbb{N}^*$ we consider $\mathbf{u}^{N,ex} $ and $\mathbf{f}^{N,ex}$ the discretizations of $u$ and $f$ defined by \eqref{discretization}.

\medskip

We begin proving the consistency of the scheme in the center. We introduce an integer $j$ such that $\tau(d)<j<N+1-\tau(d)$. Then, we do the estimation of consistency with a Taylor Lagrange formula
\begin{align*}
& \left| \left( \mathbf{D}_N \mathbf{u}^{N,ex}\right)_j - h^2 \left( \mathbf{S}_N \mathbf{f}^{N,ex} \right)_j  \right| \\
&= \sum_{i\in \mathbb{Z}} d_{i} u(x^N_{i+j}) + h^2 s_i u''(x^N_{i+j}) \\
&= \sum_{i\in \mathbb{Z}} d_{i}  p_j(x^N_i) + h^2 s_i p_j''(x^N_i) +   \sum_{i\in \mathbb{Z}} d_{i} (\xi^{N,1}_{i,j} - x_i^N)^{n+2} \frac{u^{(n+2)}(\xi^{N,1}_{i,j}) }{(n+2)!}+ h^2 s_i  (\xi^{N,2}_{i,j} - x_i^N)^{n} \frac{u^{(n)}(\xi^{N,2}_{i,j}) }{n!},
\end{align*}
with $\xi^{N,1}_{i,j},\xi^{N,2}_{i,j} \in ]x_i^N,x_{i+j}^N[$ and
\[ p_j(X) = \sum_{k=0}^{ n+1 } \frac{u^{(k)}(x_j^N)}{k!} X^{k}.\]
However, we have proved in Lemma \ref{lem pour avoir la consistance, il suffit de s'exiter sur les polynomes}, that the polynomial part of this sum is zero. Consequently, it is enough to estimate the second part. Finally, we get
\begin{align*}
 \left| \left( \mathbf{D}_N \mathbf{u}^{N,ex}\right)_j - h^2 \left( \mathbf{S}_N \mathbf{f}^{N,ex} \right)_j  \right| &\leq h^{n+2} \|u^{(n+2)}\|_{L^{\infty}(0,1)} \sum_{i\in \mathbb{Z}} |d_{i}| \frac{\tau(d)^{n+2}}{(n+2)!} + |s_i| \frac{\tau(s)^{n}}{n!} .
\end{align*}

\medskip

The same type of estimations holds near the boundary and we can prove similarly that, if $\tau(d)\geq j$ or $j\geq N+1-\tau(d)$ and if $\frac{\mu h}2\leq \gamma$ then 
\begin{align*}
 \left| \left( \mathbf{D}_N \mathbf{u}^{N,ex}\right)_j - h^2 \left( \mathbf{S}_N \mathbf{f}^{N,ex} \right)_j  \right| &\leq h^{\mu+2} \|u^{(\mu+2)}\|_{L^{\infty}(-\gamma,1 + \gamma)} \max_{1 \leq k \leq \tau(d)} \sum_{i\in \mathbb{Z}} |d_{i}^k| \frac{\tau(d)^{\mu+2}}{(\mu+2)!} + |s_i^k| \frac{\tau(s^k)^{\mu}}{\mu!} .
\end{align*}

\subsection{Proof of Lemma \ref{lemma_op_borne}}
\label{proof_lemma_op_borne}
To prove this lemma, we need to use the notations introduced to define formally $\mathbf{D}_N(d)$ in Definition \ref{def formal D_N}. 

\medskip

Now, for all $p\in \mathbb{Z}$ and for all $N\in \mathbb{N}^*$, we introduce an operator $O_{p,N}$ on $\mathscr{E}_N$ defined by
\[ \forall w \in \mathscr{E}_N, \  O_{p,N}w =  \frac12 T_{ \mathbb{1}_{\{ p,-p \}} }w = \left(  \frac{w_{i+p} + w_{i-p}}2 \right)_{i\in \mathbb{Z}} .\]
A straightforward calculation shows that the spectral decomposition of $O_{p,N} $ is
\begin{equation}
\label{dec_spec_O}
\forall k\in \mathbb{Z}, \ O_{p,N} e_{k,N} = \cos(p \pi k h ) e_{k,N} \textrm{ with } e_{k,N} = (\sin(k \pi hj))_{j\in \mathbb{Z}}.
\end{equation}

\medskip

Let $z \in \mathbb{C}\setminus [-1,1]$ be a complex number. Since the periodic function $x\mapsto (\cos(x)- z)^{-1}$ is real analytic, its Fourier transform is summable. More precisely, there exists $(c_p(z))\in l^1(\mathbb{N})$ such that
\[\forall x\in \mathbb{R}, \ \frac1{\cos(x)- z} = \sum_{p \in \mathbb{N}} c_p(z) \cos(p x).\]
Since $(c_p)$ is summable, it follows from \eqref{dec_spec_O} that $O_{1,N}-zI_{\mathscr{E}_N}$ is invertible and
\[ (O_{1,N}-zI_{\mathscr{E}_N})^{-1} =  \sum_{p \in \mathbb{N}} c_p(z) O_{p,N}. \] 
Furthermore, if $w \in \mathscr{E}_N$ then for all $p\in \mathbb{N}$ 
\[ \|O_{p,N} w\|_{l^{\infty}(\mathbb{Z})} \leq \|w\|_{l^{\infty}(\mathbb{Z})}.\]
 As a consequence, we have
\[  \|(O_{1,N}-zI_{\mathscr{E}_N})^{-1}w\|_{l^{\infty}(\mathbb{Z})} \leq \sum_{p \in \mathbb{N}} |c_p(z)| \|O_{p,N}w\|_{l^{\infty}(\mathbb{Z})} \leq \|(c_p(z))\|_{l^1(\mathbb{N})} \|w\|_{l^{\infty}(\mathbb{Z})}. \] 
To finish the proof of Lemma \ref{lemma_op_borne}, it is enough to see that
\[ \matr_{\mathscr{B}_N} O_{1,N} = \mathbf{I}_N- \frac12 \mathbf{A}_N \textrm{ and } \forall w \in \mathscr{E}_N, \ \|\matr_{\mathscr{B}_N}  w\|_{\infty} = \|w\|_{l^{\infty}(\mathbb{Z})}.  \]

\subsection{Proof of Theorem \ref{le_theoreme_diophantien}}
\label{proof of the Diophantine criterion}
It follows of the spectral decomposition of $\mathbf{A}_N$ \eqref{spectral_decompostion}, that $(\mathbf{e}^N_k)_{k=1\dots N}$ is an orthogonal basis of $\mathbb{C}^N$. Furthermore, a straightforward calculation shows that, if $k\in \llbracket 1,N \rrbracket $ then we have
\[ \| \mathbf{e}^N_k \|_2 = \sum_{j=1}^N |(\mathbf{e}^N_k)_j|^2 =  \sum_{j=1}^N  \sin^2(\pi hk j) =  \frac12 \sum_{j=1}^N 1-\cos(2\pi hk j)    = \frac{1}{2h}. \]
Consequently, if we take a vector $\mathbf{v}\in \mathbb{C}^N$, we get its discrete Fourier transform as
\[ \mathbf{v} =  2h \sum_{k=1}^N \mathbf{e}^N_k \sum_{j=1}^N \mathbf{v}_j  \sin(\pi hk j) .\]

\medskip

However, since the vectors $\mathbf{e}^N_k$ are eigenvectors of $\mathbf{A}_N$, there are eigenvectors of $P(\mathbf{A}_N)$ and their eigenvalues are $P(4\sin^2 \left( \frac{\pi}2 k h \right) )$. Consequently, we know from assumption $(iv)$ that $P(\mathbf{A}_N)$ is invertible and that we have
\[ P(\mathbf{A}_N)^{-1}  \mathbf{v} =   2h \sum_{k=1}^N \frac{e^N_k}{P \left( 4\sin^2 \left( \frac{\pi}2 k h \right)  \right)} \sum_{j=1}^N  \mathbf{v}_j  \sin(\pi hk j) .\]
Hence, if we do the estimation, $|\sin|\leq 1$, it comes
\[ \| P(\mathbf{A}_N)^{-1} \mathbf{v} \|_{\infty} \leq 2h \sum_{k=1}^N \frac{1}{|P \left( 4\sin^2 \left( \frac{\pi}2 k h \right)  \right)  |} \sum_{j=1}^N \|  \mathbf{v} \|_{\infty} \leq  \sum_{k=1}^N \frac{2}{|P \left( 4\sin^2 \left( \frac{\pi}2 k h \right)  \right) |} \| \mathbf{v} \|_{\infty}. \]
Consequently, to conclude the proof of Theorem \ref{le_theoreme_diophantien}, it is enough to proof that there exists a constant $c>0$ such that 
\begin{equation}
\label{the_sum}
\forall N\in \mathbb{N}^*, \ \sum_{k=1}^N \frac{2}{|P \left( 4\sin^2 \left( \frac{\pi}2 k h \right)  \right) |} \leq \frac{c}{\delta_{N+1}}.
\end{equation}

\medskip

But from the assumption $(iii)$, we know that there exists a polynomial $Q\in \mathbb{R}[X]$ such that
\begin{equation}
\label{cest_beau_detre_factoriel}
P(X) = Q(X) \prod_{\lambda \in \Lambda} (X-\lambda) \textrm{ and } \forall x\in [0,4], \  Q(x)\neq 0 . 
\end{equation}
Hence, we deduce from \eqref{cest_beau_detre_factoriel} that the following partial fraction decomposition holds
\[ \frac1{P(X)} = \frac1{Q(X)} \sum_{\lambda\in \Lambda} \frac{Q(\lambda)}{P'(\lambda)} \frac1{X-\lambda}.  \]
Consequently, to prove the estimation \eqref{the_sum}, it is enough to prove that
\begin{equation}
\label{the_sum_decomposee}
\forall \lambda\in \Lambda, \exists c>0,\forall N\in \mathbb{N}^*, \ \sum_{k=1}^N \frac1{|4\sin^2 \left( \frac{\pi}2 k h \right)-\lambda|} \leq c \frac{c}{\delta_{N+1}}.
\end{equation}

\medskip

To prove \eqref{the_sum_decomposee}, it is crucial to deduce, from the conditions $(i)$ and $(iv)$, that there exists a constant $c>0$ such that
\begin{equation}
\label{cest_quand_meme_pas_mieux_quavant}
 \forall q\in \mathbb{N}^*, \ \delta_{q} \leq \frac{c}{q^2}. 
\end{equation}
Then, it is enough, to distinguish the case $\lambda=0$ from the case $\lambda\neq 0$. On the one hand, if $\lambda=0$, using \eqref{cest_quand_meme_pas_mieux_quavant}, we have
\[  \sum_{k=1}^N \frac1{4\sin^2 \left( \frac{\pi}2 k h \right)} \leq  \sum_{k=1}^N \frac1{4 \left(  k h \right)^2} \leq \frac{\pi^2}6 \frac1{4h^2} \leq \frac{\pi^2}{24} \frac{c}{\delta_{N+1}}.   \]
On the other hand, $x\mapsto 4\sin^2 \left( \frac{\pi}2 x \right) $ is a diffeomorphism from $]0,1[$ to $]0,4[$. Hence, if $\lambda \neq 0$, and  since we know from assumption $(ii)$ that $\lambda\neq 4$, there exists a constant $\widetilde{c}>0$ such that one has
\[ \forall x\in [0,1], \  |x - \widetilde{\lambda}| \leq \widetilde{c}|4\sin^2 \left( \frac{\pi}2 x \right) - \lambda |, \]
where $\widetilde{\lambda}\in ]0,1[$ is defined by
\[  4\sin^2 \left( \frac{\pi}2 \widetilde{\lambda} \right) = \lambda. \]
Since $\delta$ does not have any zero index, the assumption $(iv)$ provides $\widetilde{\lambda}\notin \mathbb{Q}$. Hence, we deduce that
\[ \forall q\in \mathbb{N}^*, \exists ! p_q\in \llbracket 0,q \rrbracket, \ |\widetilde{\lambda}-\frac{p_q}{q}|<\frac1{2q}.  \]
As a consequence, with the estimation \eqref{cest_quand_meme_pas_mieux_quavant}, we have
\begin{align*}
\sum_{k=1}^N \frac1{|4\sin^2 \left( \frac{\pi}2 k h \right)-\lambda|}  \leq \sum_{k\in   \llbracket 1,N \rrbracket \setminus\{p_{N+1}\} }  \frac{ \widetilde{c}}{| kh - \widetilde{\lambda}  |} + \frac1{| 4\sin^2 \left( \frac{\pi}2 p_{N+1} h \right)  - \lambda  |}		 &\leq   \sum_{k\in  \llbracket 1,N \rrbracket \setminus\{p_{N+1}\} } \frac{2 \widetilde{c}}{h} +   \frac1{\delta_{N+1}}\\
																							 &\leq   \frac{2 \widetilde{c}}{h^2} + \frac1{\delta_{N+1}}
																							 \leq \frac{2 \widetilde{c} c+ 1}{\delta_{N+1}}.
\end{align*}
\bibliographystyle{plain}
\bibliography{CFD_bib}
\end{document}